\documentclass[smallextended,numbook,hidelinks]{svjour3_mod}

\usepackage[a-1b]{pdfx}

\usepackage{amsmath}
\usepackage{amssymb}

\usepackage{mathptmx}
\usepackage{helvet}
\usepackage{courier}

\usepackage{graphicx}
\usepackage{booktabs}

\begin{document}

\title{An iterative method for the solution of Laplace-like\\
 equations in high and very high space dimensions}

\author{Harry Yserentant}

\institute{
 Technische Universit\"at Berlin, Institut f\"ur Mathematik,
 10623 Berlin, Germany\\
\email{yserentant@math.tu-berlin.de}}

\date{October 17, 2026}

\titlerunning{An iterative method for Laplace-like 
 equations in high dimensions} 
 
\authorrunning{Harry Yserentant}

\maketitle

\noindent
{\bf Abstract}
\medskip

\noindent
This paper deals with the equation $-\Delta u+\mu u=f$
on high-dimensional spaces $\mathbb{R}^m$, where the 
right-hand side $f(x)=F(Tx)$ is composed of a separable 
function $F$ with an integrable Fourier transform on a 
space of a dimension $n>m$ and a linear mapping given 
by a matrix $T$ of full rank and $\mu\geq 0$ is a constant. 
For example, the right-hand side can explicitly depend 
on differences $x_i-x_j$ of components of $x$. We show 
that the solution of this equation can be expanded 
into sums of functions of the same structure and develop 
in this framework an equally simple and fast iterative 
method for its computation. The method is based on the 
observation that in almost all cases and for large problem 
classes the expression $\|T^ty\|^2$ deviates on the unit 
sphere $\|y\|=1$ the less from its mean value the higher 
the dimension $m$ is, a concentration of measure effect.
The higher the dimension $m$, the faster the iteration 
converges.

%
%
%

%
%
%

%
%
%
%
%
%
%
%


\renewcommand {\thefigure}{\arabic{figure}}

\renewcommand {\thetable}{\arabic{table}}

\newcommand   {\rmref}[1]    {{\rm (\ref{#1})}}

\newcommand   {\diff}[1]     {\mathrm{d}#1}

\newcommand   {\fourier}[1]  {\widehat{#1}}

\newcommand   {\tnorm}[1]    {|\!|\!|#1|\!|\!|}

\def \P       {\mathbb{P}}

\newcommand   {\sectorx}[1]  {\P\big(\big\{x\,\big|\,#1\big\}\big)}
\newcommand   {\sectoro}[1]  {\P\big(\big\{\omega\,\big|\,#1\big\}\big)}
\newcommand   {\sectore}[1]  {\P\big(\big\{\eta\,\big|\,#1\big\}\big)}

\def \dx      {\,\diff{x}}
\def \dy      {\,\diff{y}}

\def \deta    {\,\diff{\eta}}
\def \dmu     {\,\diff{\mu}}
\def \domega  {\,\diff{\omega}}

\def \dr      {\,\diff{r}}
\def \ds      {\,\diff{s}}
\def \dt      {\,\diff{t}}

\def \e       {\mathrm{e}}
\def \i       {\mathrm{i}}

\def \B       {\mathcal{B}}
\def \L       {\mathcal{L}}

\def \a       {\alpha}
\def \b       {\beta}


\section{Introduction}
\label{sec1}

\setcounter{equation}{0}
\setcounter{theorem}{0}

The numerical solution of partial differential 
equations in high space dimensions is a difficult 
and challenging task. Methods such as finite 
elements, which work perfectly in two or three 
dimensions, are not suitable for solving such 
problems because the effort grows exponentially 
with the dimension. Random walk based techniques 
only provide solution values at selected points. 
Sparse grid methods are best suited for problems 
in still moderate dimensions. Tensor-based methods 
\cite{Bachmayr}, \cite{Hackbusch_1}, 
\cite{Khoromskij} stand out in this area. They 
are not subject to such limitations and perform 
surprisingly well in a large number of cases. 
Tensor-based methods exploit the structure of
the solution rather than its regularity.
Consider the equation
\begin{equation}    \label{eq1.1}
-\Delta u+\mu u=f
\end{equation}
on $\mathbb{R}^m$ for high dimensions $m$, 
where $\mu>0$ is a given constant. Provided the 
right-hand side $f$ of the equation (\ref{eq1.1}) 
possesses an integrable Fourier transform,  
\begin{equation}    \label{eq1.2}
u(x)=\left(\frac{1}{\sqrt{2\pi}}\right)^m\!\int
\frac{1}{\|\omega\|^2+\mu}\,\fourier{f}(\omega)\,
\e^{\,\i\,\omega\cdot x}\domega
\end{equation}
is a solution of this equation, and the only solution
that tends uniformly to zero as $x$ goes to infinity.
If the right-hand side $f$ of the equation is a 
tensor product
\begin{equation}    \label{eq1.3}
f(x)=\prod_i\phi_i(x_i)
\end{equation}
of functions, say from the three-dimensional 
space to the real numbers, or a sum of such 
tensor products, the same holds for the 
Fourier transform of $f$. If one replaces 
the corresponding term in the high-dimensional 
integral (\ref{eq1.2}) by an approximation 
\begin{equation}    \label{eq1.4}
\frac{1}{\|\omega\|^2+\mu}\approx
\sum_k a_k\,\e^{-\beta_k\left(\|\omega\|^2+\mu\right)}
= \sum_k a_k\,\e^{-\beta_k\mu}
\prod_i\e^{-\beta_k\|\omega_i\|^2}
\end{equation}
based on an appropriate approximation of $1/r$ 
by a sum of exponential functions, the integral 
then collapses to a sum of products of 
lower-dimensional integrals. That is, the 
solution can be approximated by a sum of such 
tensor products whose number is independent 
of the space dimension. The computational 
effort no longer increases exponentially, 
but only linearly with the space dimension.

However, the right-hand side of the equation does
not always have such a simple structure and 
cannot always be well represented by tensors of 
low rank. A prominent example is quantum mechanics.  
The potential in the Schr\"odinger equation depends 
on the distances between the particles considered.
Therefore, it is desirable to approximate the 
solutions of this equation by functions that 
explicitly depend on the position of the particles 
relative to each other. As a building block in more 
comprehensive calculations, this can require the 
solution of equations of the form (\ref{eq1.1}) 
with right-hand sides that are composed of terms 
such as
\begin{equation}    \label{eq1.5}
f(x)=\bigg(\prod_i\phi_i(x_i)\bigg)
\bigg(\prod_{i<j}\phi_{ij}(x_i-x_j)\bigg).
\end{equation}
The question is whether such structures transfer 
to the solution and whether in such a context 
arising iterates stay in this class. The present 
work deals with this problem. We present a 
conceptually simple iterative method that 
preserves such structures and takes advantage 
of the high dimensions.

First we embed the problem as in our former paper 
\cite{Yserentant_2020} into a higher dimensional 
space introducing, for example, some or all 
differences $x_i-x_j$, $i<j$, in addition to the 
components $x_i$ of the vector $x\in\mathbb{R}^m$ 
as additional variables. We assume that the 
right-hand side of the equation (\ref{eq1.1}) is 
of the form $f(x)=F(Tx)$, where $T$ is a matrix 
of full rank that maps the vectors in $\mathbb{R}^m$ 
to vectors in an $\mathbb{R}^n$ of a still higher 
dimension and $F:\mathbb{R}^n\to\mathbb{R}$ 
is a function that possesses an integrable Fourier 
transform and as such is continuous. The solution 
of the equation (\ref{eq1.1}) is then the trace 
$u(x)=U(Tx)$ of the then equally continuous 
function
\begin{equation}    \label{eq1.6}
U(y)=\left(\frac{1}{\sqrt{2\pi}}\right)^n\!\int
\frac{1}{\|T^t\omega\|^2+\mu}\,\fourier{F}(\omega)\,
\e^{\,\i\,\omega\cdot y}\domega,
\end{equation}
which is, in a corresponding sense, the solution 
of a degenerate elliptic equation 
\begin{equation}    \label{eq1.7}
\L U+\mu U=F.
\end{equation} 
This equation replaces the original equation 
(\ref{eq1.1}). Its solution (\ref{eq1.6}) 
is approximated by the iterates arising from 
a polynomially accelerated version of the 
basic method
\begin{equation}    \label{eq1.8}     
U_{k+1}=\,U_k\,-\,(-\Delta+\mu)^{-1}(\L U_k+\mu U_k-F)
\end{equation}
starting from $U_0=0$. The calculation of the 
iterates requires the solution of equations of 
the form (\ref{eq1.1}), that is, the calculation 
of integrals of the form (\ref{eq1.2}), now over 
the higher dimensional $\mathbb{R}^n$. The 
symbol $\|T^t\omega\|^2$ of the operator $\L$ 
is a homogeneous second-order polynomial in 
$\omega$. For separable right-hand sides $F$ 
as above, the calculation of these integrals 
thus reduces to the calculation of products 
of lower, in the extreme case one-dimensional 
integrals.

The reason for the usually astonishingly fast 
convergence of this iteration is the directional 
behavior of the term $\|T^t\omega\|^2$. The higher 
the dimension  $m$, the lower the probability that 
the values $\|T^t\eta\|^2$ for $\eta$ on the unit 
sphere $S^{n-1}$ of the $\mathbb{R}^n$ deviate 
much from their mean, a typical concentration 
of measure effect. To capture this phenomenon 
quantitatively, we introduce the probability 
measure
\begin{equation}    \label{eq1.9}
\P(M)=\frac{1}{n\nu_n}\int_{M\cap S^{n-1}}\!\deta
\end{equation}
on the Borel subsets $M$ of the $\mathbb{R}^n$, 
where $\nu_n$ is the volume of the unit ball 
and $n\nu_n$ thus is the area of the unit sphere.
If $M$ is a subset of the unit sphere, $\P(M)$ 
is equal to the ratio of the area of M to the 
area of the unit sphere. If $M$ is a sector, 
that is, if $M$ contains with a vector $\omega$ 
also its scalar multiples, $\P(M)$ measures the 
opening angle of $M$. The quantity on which all 
our analysis is based is the angular distribution
\begin{equation}    \label{eq1.10}
F(\delta)={\P\big(\big\{\omega\in\mathbb{R}^n\,\big|\,
\|T^t\omega\|^2\leq\delta\|\omega\|^2\big\}\big)}
\end{equation}
of the values $\|T^t\omega\|^2$. We assume 
that its expected value, the mean value of 
the expression $\|T^t\eta\|^2$ over the unit 
sphere, is one. This is only a matter of the 
scaling of the variables in the higher 
dimensional space and does not represent a 
restriction. Apart from extreme cases, the 
distribution (\ref{eq1.10}) approaches a 
normal distribution with increasing dimensions. 
The concentration of measure effect is reflected 
in the fact that the variance of the distribution 
decreases and tends, under rather general 
circumstances, to zero as the dimensions 
increase.

For a given $\rho<1$, let $S$ be the sector that 
consists of the points $\omega$ for which the 
expression $\|T^t\omega\|^2$ differs from 
$\|\omega\|^2$ by $\rho\|\omega\|^2$ or less. 
If the Fourier transform of the right-hand side 
of the equation (\ref{eq1.7}) vanishes at all 
$\omega$ outside this set, the same holds for 
the Fourier transform of its solution (\ref{eq1.6}) 
and the Fourier transforms of the iterates $U_k$. 
Under this condition, the iteration error 
decreases at least like 
\begin{equation}    \label{eq1.11}
\|U-U_k\|\leq\rho^k\|U\|
\end{equation}
with respect to a broad range of Fourier-based 
norms. This is admittedly an idealized 
situation and the actual convergence behavior 
is more complicated. Nevertheless, this example 
accurately describes what to expect. Provided 
that for points $\eta$ on the unit sphere
the values $\|T^t\eta\|^2$ are approximately 
normally distributed with expected value $E=1$ 
and small variance $V$, the measure (\ref{eq1.9}) 
of the sector $S$ is almost one as soon as $\rho$ 
exceeds the standard deviation $\sigma=\sqrt{V}$ 
by more than a moderate factor. The sector $S$ 
then fills almost the entire frequency space. 
The higher the dimensions and the smaller the 
variance, the faster the iterates approach 
the solution.

The rest of this paper is organized as follows. 
Section~\ref{sec2} sets the framework and is  
devoted to the representation of the solutions 
of the equation (\ref{eq1.1}) as traces of higher 
dimensional functions (\ref{eq1.6}) for 
right-hand sides that are themselves traces of 
functions with an integrable Fourier transform. 
In comparison with the proof in 
\cite{Yserentant_2020}, we give a more direct 
proof of this representation. In addition, we 
introduce two scales of norms with respect to 
which we later estimate the iteration error.

In a sense, the following section forms the core 
of the present work. It is devoted to the study
of the angular distribution (\ref{eq1.10}) of 
the values $\|T^t\omega\|^2$. Our first result, 
Theorem~\ref{thm3.1}, is a semi-explicit 
representation of the density of this distribution 
in form of an integral over the unit sphere in 
$\mathbb{R}^m$. It gives detailed information
about the behavior of this distribution for small 
$\delta$. In a special case, when all singular 
values of $T$ are equal, it becomes a rescaled 
beta distribution. Moreover, we show how the 
expected value $E$ and the variance $V$ of the 
distribution can be expressed in terms of the 
singular values of the matrix $T$. Using this 
representation, we show that for random matrices 
$T$ with assigned expected value $E=1$ the 
expected value of the variances $V$ not only 
tends to zero as the dimension $m$ goes to 
infinity, but also that these variances cluster 
the more around their expected value the larger 
$m$ is. We also study a class of matrices $T$ 
that are associated with interaction graphs. 
These matrices correspond to the case that some 
or all coordinate differences are introduced as 
additional variables and formed the motivation 
for the present work. The expected values 
that are assigned to these matrices take the 
value one and the variances $V$ can be 
expressed directly in terms of the vertex 
degrees. Finally, we show how to efficiently 
sample the values $\|T^t\eta\|^2$ for a large 
number of points $\eta$ that are uniformly 
distributed on the unit sphere. Calculations 
of this kind support our claim  that these 
values are approximately normally distributed 
in high dimensions.

In the final section, we return to the
higher-dimensional counterpart (\ref{eq1.7}) of 
the original equation (\ref{eq1.1}) and examine
the convergence behavior of the polynomially 
accelerated version of the iteration (\ref{eq1.8}) 
for its solution. Special attention is given to 
the limit case $\mu=0$ of the Laplace equation. 
The section ends with a brief review of an 
approximation of the form (\ref{eq1.4}) by 
sums of Gauss functions.


\setcounter{section}{1}

\section{Solutions as traces of higher-dimensional functions}
\label{sec2}

\setcounter{equation}{0}
\setcounter{theorem}{0}
 
In this paper we are mainly concerned with 
functions $U:\mathbb{R}^n\to\mathbb{R}$, $n$ a 
potentially high dimension, that possess the
then unique representation
\begin{equation}    \label{eq2.1}
U(y)=\left(\frac{1}{\sqrt{2\pi}}\right)^n
\!\int\fourier{U}(\omega)\,
\e^{\,\i\,\omega\cdot y}\domega
\end{equation}
in terms of an integrable function $\fourier{U}$, 
their Fourier transform. Such functions are 
uniformly continuous by the Riemann-Lebesgue theorem 
and vanish at infinity. The space $\B_0(\mathbb{R}^n)$ 
of these functions becomes a Banach space under 
the norm
\begin{equation}    \label{eq2.2}   
\|U\|_0=\int|\fourier{U}(\omega)|\domega.
\end{equation}
Let $T$ be an arbitrary $(n\times m)$-matrix 
of full rank $m<n$ and let
\begin{equation}    \label{eq2.3}
u:\mathbb{R}^m\to\mathbb{R}:x\to U(Tx)
\end{equation}
be the trace of a function in $U\in\B_0(\mathbb{R}^n)$.
Since the functions in $\B_0(\mathbb{R}^n)$ are uniformly 
continuous, the same also applies to the traces of these 
functions. Since there is a constant~$c$ with 
$\|x\|\leq c\,\|Tx\|$ for all $x\in\mathbb{R}^m$, the 
trace functions (\ref{eq2.3}) vanish at infinity as 
$U$ itself. The next lemma gives a criterion for the 
existence of partial derivatives of the trace functions, 
where we use the common multi-index notation. 

\begin{lemma}       \label{lm2.1}
Let $U:\mathbb{R}^n\to\mathbb{R}$ be a function 
in $\B_0(\mathbb{R}^n)$ and let the functions
\begin{equation}    \label{eq2.4}
\omega\to (\i\,T^t\omega)^\beta\fourier{U}(\omega),
\quad \beta\leq\alpha,
\end{equation}
be integrable. Then the trace function \rmref{eq2.3} 
possesses the partial derivative
\begin{equation}    \label{eq2.5}
(\mathrm{D}^\alpha u)(x)=
\left(\frac{1}{\sqrt{2\pi}}\right)^n
\!\int(\i\,T^t\omega)^\alpha\fourier{U}(\omega)\,
\e^{\,\i\,\omega\cdot Tx}\domega
\end{equation}
that, like $u$, is itself uniformly continuous 
and vanishes at infinity.
\end{lemma}

\begin{proof}
Let $e_k\in\mathbb{R}^m$ be the vector with the
components $e_k|_j=\delta_{kj}$. To begin with, 
we examine the limit behavior of the difference 
quotient
\begin{displaymath}
\frac{u(x+he_k)-u(x)}{h} =
\left(\frac{1}{\sqrt{2\pi}}\right)^n\!\int
\frac{\e^{\,\i\,h\omega\cdot Te_k}-1}{h}\;
\fourier{U}(\omega)\,
\e^{\,\i\,\omega\cdot Tx}\domega
\end{displaymath}
of the trace function as $h$ goes to zero.
Because of
\begin{displaymath}
\left|\,\frac{\e^{\,\i\,ht}-1}{h}\,\right|\leq\,|\,t\,|,
\quad
\lim_{h\to 0}\frac{\e^{\,\i\,ht}-1}{h}=\,\i\,t,
\end{displaymath}
and under the condition that the function 
$\omega\to\omega\cdot Te_k\,\fourier{U}(\omega)$
is integrable, it tends to 
\begin{displaymath}
(\mathrm{D}_ku)(x) = 
\left(\frac{1}{\sqrt{2\pi}}\right)^n
\!\int\i\,\omega\cdot Te_k\,\fourier{U}(\omega)\,
\e^{\,\i\,\omega\cdot Tx}\domega
\end{displaymath}
as follows from the dominated convergence theorem.
Because of $\omega\cdot Te_k=T^t\omega\cdot e_k$,
this proves (\ref{eq2.5}) for partial derivatives
of order one. For partial derivatives of higher
order, the proposition follows by induction.
\qed 
\end{proof}
Let $D(\L)$ be the space of the functions 
$U\in\B_0(\mathbb{R}^n)$ with finite (semi)-norm
\begin{equation}    \label{eq2.6}
\|U\|_2=\int\|T^t\omega\|^2\,|\fourier{U}(\omega)|\domega.
\end{equation}
Because of 
$|(T^t\omega)^\beta|\leq 1+\|T^t\omega\|^2$
for all multi-indices $\beta$ of order two 
or less, the traces of the functions in this 
space are twice continuously differentiable
by Lemma~\ref{lm2.1}. Let 
$\L:D(\L)\to\B_0(\mathbb{R}^n)$ be the  
pseudo-differential operator given by
\begin{equation}    \label{eq2.7}
(\L U)(y)=\left(\frac{1}{\sqrt{2\pi}}\right)^n
\!\int\|T^t\omega\|^2\,\fourier{U}(\omega)\,
\e^{\,\i\,\omega\cdot y}\domega.
\end{equation}
For the functions $U\in D(\L)$ and their 
traces (\ref{eq2.3}), by Lemma~\ref{lm2.1}
\begin{equation}    \label{eq2.8}
-\,(\Delta u)(x)=(\L U)(Tx)
\end{equation}
holds. With corresponding right-hand sides,
the solutions of the equation (\ref{eq1.1})
are thus the traces of the solutions 
$U\in D(\L)$ of the pseudo-differential 
equation
\begin{equation}    \label{eq2.9}
\L U+\mu U=F.
\end{equation}

\begin{theorem}     \label{thm2.1}
Let $F:\mathbb{R}^n\to\mathbb{R}$ be a function with 
integrable Fourier transform, let $f(x)=F(Tx)$, and 
let $\mu$ be a positive constant. Then the trace 
\rmref{eq2.3} of the function
\begin{equation}    \label{eq2.10}
U(y)=\left(\frac{1}{\sqrt{2\pi}}\right)^n\!\int
\frac{1}{\|T^t\omega\|^2+\mu}\,\fourier{F}(\omega)\,
\e^{\,\i\,\omega\cdot y}\domega
\end{equation}
is twice continuously differentiable 
and the only solution of the equation
\begin{equation}    \label{eq2.11}
-\Delta u+\mu u=f
\end{equation}
whose values tend uniformly to zero as 
$\|x\|$ goes to infinity. Provided the 
function 
\begin{equation}    \label{eq2.12}
\omega\,\to\,
\frac{1}{\|T^t\omega\|^2}\,\fourier{F}(\omega)\,
\end{equation}
is integrable, the same holds
in the limit case $\mu=0$.
\end{theorem}

\begin{proof}
That the trace $u$ is a classical solution 
of the equation (\ref{eq2.11}) follows from 
the remarks above, and that $u$ vanishes at 
infinity by the already discussed reasons
from the Riemann-Lebesgue theorem. The 
maximum principle ensures that no further 
solution of the equation (\ref{eq2.11}) 
exists that vanishes at infinity.
\qed
\end{proof}

From now on, the equation (\ref{eq2.9}) will  
replace the original equation (\ref{eq2.11}).
Our aim is to compute its solution 
(\ref{eq2.10}) iteratively by polynomial
accelerated versions of the basic iteration
(\ref{eq1.8}). The convergence properties 
of this iteration depend decisively on
the directional behavior of the values
$\|T^t\omega\|^2$, which will be studied in 
the next section, before we return to the 
equation and its iterative solution.

Before we continue with these considerations 
and turn our attention to the directional 
behavior of these values, we introduce 
the norms with respect to which we will 
show convergence. The starting point is the 
radial-angular decomposition 
\begin{equation}    \label{eq2.13}
\int_{\mathbb{R}^n}f(x)\dx\,= 
\int_{S^{n-1}}\bigg(\int_0^\infty\!f(r\eta)r^{n-1}\dr\bigg)\deta
\end{equation}
of the integrals of functions in $L_1(\mathbb{R}^n)$ 
into an inner radial and an outer angular part. 
Inserting the characteristic function of the unit 
ball, one recognizes that the area of the 
$n$-dimensional unit sphere $S^{n-1}$ is $n\nu_n$, 
with $\nu_n$ the volume of the unit ball. If $f$ is 
rotationally symmetric, $f(r\eta)=f(re)$ holds for 
every $\eta\in S^{n-1}$ and every fixed, arbitrarily 
given unit vector $e$.  In this case, 
(\ref{eq2.13}) simplifies to
\begin{equation}    \label{eq2.14}
\int_{\mathbb{R}^n}f(x)\dx\,=\,n\nu_n\int_0^\infty\!f(re)r^{n-1}\dr.
\end{equation}

The norms split into two groups and are labeled 
by a smoothness parameter. Like the differential 
operator (\ref{eq2.7}) itself, they depend on 
the matrix $T$ under consideration. In a first 
step, we assign the function
$\phi:S^{n-1}\to\mathbb{R}$ with the values
\begin{equation}    \label{eq2.15}
\phi(\eta)=n\nu_n\int_0^\infty\|T^t\eta\|^s\,|\fourier{U}(r\eta)|\,r^{s+n-1}\dr
\end{equation}
to a given integrable function $U:\mathbb{R}^n\to\mathbb{R}$
and a given real smoothness parameter $s$. 
The two norms of $U$ are then defined via 
\begin{equation}    \label{eq2.16}
\|U\|_s=\frac{1}{n\nu_n}\int_{S^{n-1}}\phi(\eta)\deta, \quad
\tnorm{U}_s^2=\frac{1}{n\nu_n}\int_{S^{n-1}}\phi(\eta)^2\deta.
\end{equation}
If the norm $\tnorm{U}_s$ of $U$ is finite, then 
the norm $\|U\|_s$ is also finite. In this case, 
the estimate $\|U\|_s\leq\tnorm{U}_s$ holds.  
This follows from the Cauchy-Schwarz inequality. 
In cartesian coordinates, the first of the two
 norms takes the simpler form 
\begin{equation}    \label{eq2.17}
\|U\|_s=\int_{\mathbb{R}^n}\|T^t\omega\|^s\,|\fourier{U}(\omega)|\domega.
\end{equation}
The traces $u(x)=U(Tx)$ of functions $U$ with finite 
norms $\|U\|_0$ and $\|U\|_s$ for a real number 
$s\geq 1$ possess continuous partial derivatives 
of order $k$, $k\leq s$, which can be bounded in 
terms of the then also finite norms $\|U\|_k$ and 
vanish at infinity.

A direct consequence of the definitions are 
the following regularity theorems.

\begin{theorem}     \label{thm2.2}
For right-hand sides $F$ with finite norms $\|F\|_s$ 
and $\tnorm{F}_s$, respectively, the solution 
\rmref{eq2.10} of the equation \rmref{eq2.9} 
satisfies the estimates
\begin{equation}    \label{eq2.18}
\|U\|_{s+2}\leq\|F\|_s, \quad 
\tnorm{U}_{s+2}\leq\tnorm{F}_s
\end{equation}
for all smoothness parameters $s$, independent 
of $\mu\geq 0$. If $\mu=0$, equality holds.
\end{theorem}
This simply results from the Fourier representation
of the solution and the estimate
\begin{equation}    \label{eq2.19}
\frac{1}{r^2\|T^t\eta\|^2+\mu}\leq
\frac{1}{\|T^t\eta\|^2}\,\frac{1}{r^2}.
\end{equation}
If $\mu>0$, the left-hand side can be estimated 
by $1/\mu$ instead. This leads to

\begin{theorem}     \label{thm2.3}
For right-hand sides $F$ with finite norms $\|F\|_s$ 
and $\tnorm{F}_s$, respectively, the solution 
\rmref{eq2.10} of the equation \rmref{eq2.9} 
satisfies the estimates
\begin{equation}    \label{eq2.20}
\|U\|_s\leq\frac{1}{\mu}\,\|F\|_s, \quad 
\tnorm{U}_s\leq\frac{1}{\mu}\,\tnorm{F}_s
\end{equation}
for all smoothness parameters $s$, provided 
$\mu$ is strictly positive.
\end{theorem}


\setcounter{section}{2}

\section{The angular distribution}
\label{sec3}

\setcounter{equation}{0}
\setcounter{theorem}{0}

\setcounter{figure}{0}

The purpose of this section is a very detailed  
study of the angular distribution of the values 
$\|T^t\omega\|^2$, based on our previously 
introduced probability measure
\begin{equation}    \label{eq3.1}
\P(M)=\frac{1}{n\nu_n}\int_{M\cap S^{n-1}}\!\deta
\end{equation}
on the Borel subsets $M$ of the $\mathbb{R}^n$
that measures the opening angles of sectors. 
In terms of this probability measure, the
angular distribution of these values is
\begin{equation}    \label{eq3.2}
F(\delta)=\P\big(\big\{\omega\in\mathbb{R}^n\,\big|\,
\|T^t\omega\|^2\leq\delta\|\omega\|^2\big\}\big)
\end{equation}
or, after restriction to the unit sphere 
$\|\eta\|=1$ itself, 
\begin{equation}    \label{eq3.3}
F(\delta)=\P\big(\big\{\eta\in S^{n-1}\,\big|\,
\|T^t\eta\|^2\leq\delta\big\}\big).
\end{equation}
The direct calculation of such and similar 
quantities is difficult. We therefore reduce 
the calculation of corresponding integrals 
over the unit sphere to the calculation of 
simpler volume integrals. Our main tool is the 
radial-angular decomposition (\ref{eq2.13}).

\begin{lemma}       \label{lm3.1}
If the function $\chi:\mathbb{R}^n\to\mathbb{R}$ 
is positively homogeneous of nonnegative 
degree $\ell$, the integral of $\chi$ over the 
unit sphere is equal to the volume integral
\begin{equation}    \label{eq3.4}
\frac{1}{n\nu_n}\int_{S^{n-1}}\chi(\eta)\deta =
C(\ell)\int_{\mathbb{R}^n}\chi(\omega)W(\omega)\domega,
\end{equation}
where the rotationally symmetric function 
$W=W_n$ is the normed Gauss function
\begin{equation}    \label{eq3.5}
W(\omega)=\left(\frac{1}{\sqrt{\pi}}\right)^n
\exp\big(-\|\omega\|^2\big),
\end{equation}
which splits into a product of one-dimensional 
functions, and the prefactor is
\begin{equation}    \label{eq3.6}
C(\ell)=\frac{\Gamma(n/2)}{\Gamma((n+\ell)/2)}.
\end{equation}
\end{lemma}

\begin{proof}
This follows immediately from the radial angular 
decomposition of volume integrals into an inner 
radial and an outer angular part, the identities
\begin{displaymath} 
\nu_n=\frac{2}{n}\,\frac{\pi^{n/2}}{\Gamma(n/2)},
\quad \int_0^\infty t^j\e^{-t^2}\dt =
\frac{1}{2}\,\Gamma\left(\frac{j+1}{2}\right),
\end{displaymath}
and the homogeneity $\chi(r\eta)=r^\ell\chi(\eta)$ 
of the function under consideration.
\qed
\end{proof}

Let $H$ be the Heaviside function, with the values 
$H(t)=0$ for $t<0$ and $H(t)=1$ for $t\geq 0$, and let 
$\chi(\omega,\delta)=H(\delta\|\omega\|^2-\|T^t\omega\|^2)$. 
As $\chi(\omega,\delta)$ is homogeneous of degree zero
as a function of $\omega$, the distribution (\ref{eq3.2})
possesses the representation
\begin{equation}    \label{eq3.7}
F(\delta)=
\int_{\mathbb{R}^n}\chi(\omega,\delta)W(\omega)\domega
\end{equation}
by Lemma~\ref{lm3.1}. It depends only on the singular 
values of the matrix $T$. This can be shown by means 
of the singular value decomposition of $T$ and the 
transformation theorem for volume integrals. Since the 
expression $\chi(\omega,\delta)$ is right-continuous 
as a function of $\delta$, the distribution is 
right-continuous by the dominated convergence theorem. 
For dimensions $n>m+2$, it possesses a density and 
has a representation
\begin{equation}    \label{eq3.8}
F(\delta)=\int_0^\delta\!f(t)\dt.
\end{equation}
Near $t=0$, it shows a characteristic, 
dimension-dependent behavior. 

\begin{theorem}     \label{thm3.1}
If $n-m>2$, the distribution \rmref{eq3.2} possesses 
the density function $f$ that vanishes for $t\leq 0$ 
and is, for arguments $t>0$, given by
\begin{equation}    \label{eq3.9}
f(t)=\frac{1}{B(\a,\b)}
\frac{1}{m\nu_m}\int_{S^{m-1}}\!
\frac{t^{\a-1}R(1-\,t\,\|\Sigma_0^{-1}\eta\|^2)^{\b-1}}
{\det\Sigma_0}\,\deta, 
\end{equation}
where $\Sigma_0=\mathrm{diag}(\sigma_1,\ldots,\sigma_m)$ 
is the $(m\times m)$-diagonal matrix with the singular 
values of the matrix $T$ on its diagonal, the 
coefficients $\a$ and $\b$ are
\begin{equation}    \label{eq3.10}
\a=\frac{m}{2}, \quad \b=\frac{n-m}{2},
\end{equation}
the function $R$ takes the values $R(t)=\max(0,t)$, 
and $B$ is Euler's beta function. The density 
attains values $f(t)>0$ for $0<t<\|T\|^2$ and 
vanishes for $t\geq\|T\|^2$. 
\end{theorem}

\begin{proof}
As the distribution depends only on the singular 
values of the matrix $T$, we can assume that
$T^t=(\Sigma_0\;0)$ is itself a diagonal matrix. 
In the following, we split the vectors 
$\omega\in\mathbb{R}^n$ into parts 
$x\in\mathbb{R}^m$ and $y\in\mathbb{R}^{n-m}$. 
Let $a(x,\delta)\geq 0$, $\delta>0$, be given by
\begin{displaymath}
a(x,\delta)^2=
R\bigg(\frac{\|\Sigma_0 x\|^2-\delta\|x\|^2}{\delta}\bigg).
\end{displaymath}
Since $\|\Sigma_0 x\|^2\leq\delta(\|x\|^2+\|y\|^2)$
holds if and only if $\|y\|-a(x,\delta)\geq 0$,
the distribution can be written as a double 
integral as follows
\begin{displaymath}
F(\delta)=\int_{\mathbb{R}^m}\bigg(\int_{\mathbb{R}^{n-m}}
H\big(\|y\|-a(x,\delta)\big)W_{n-m}(y)\dy\bigg)W_m(x)\dx.
\end{displaymath}
This results from its representation (\ref{eq3.7})
and Fubini's theorem. By (\ref{eq2.14}), the inner 
integral reduces to the one-dimensional integral
\begin{displaymath}
\frac{2}{\Gamma((n-m)/2)}\int_0^\infty 
H\big(r-a(x,\delta)\big)\e^{-r^2}r^{(n-m)-1}\dr.
\end{displaymath}
Introducing the function
\begin{displaymath}
\phi(s)=\frac{2}{\Gamma((n-m)/2)}
\int_s^\infty r^{(n-m)-1}\e^{-r^2}\dr,
\end{displaymath}
the distribution therefore takes the form
\begin{displaymath}
F(\delta)=\int_{\mathbb{R}^m}\phi(a(x,\delta))W_m(x)\dx.
\end{displaymath}
Let $x\neq 0$. Because 
$a(x,t)=(\|\Sigma_0 x\|^2/t-\|x\|^2)^{1/2}>0$ 
for $0<t<\|\Sigma_0 x\|^2/\|x\|^2$, the function 
$t\to\phi(a(x,t))$ is differentiable on this 
interval. There it has the derivative
\begin{displaymath}
h(x,t)=\frac 
{a(x,t)^{(n-m)-2}\exp(-\,a(x,t)^2)\|\Sigma_0 x\|^2}
{\Gamma((n-m)/2)\,t^2}.
\end{displaymath}
Because $n-m>2$, this derivative tends to zero as $t$ 
goes to $\|\Sigma_0 x\|^2/\|x\|^2$ from the left.
Because $a(x,t)=0$ for $\|\Sigma_0 x\|^2/\|x\|^2\leq t$, 
$\phi(a(x,t))$ is therefore continuously differentiable 
as a function of $t$ on the entire positive real axis 
$t>0$ and its derivative has the above representation. 
The same applies in the case $x=0$. For 
$0<\delta_0<\delta$
\begin{displaymath}
F(\delta)-F(\delta_0)=\int_{\mathbb{R}^m}
\bigg(\int_{\delta_0}^\delta h(x,t)\dt\bigg)W_m(x)\dx
\end{displaymath}
follows, or, after interchanging the order 
of integration,
\begin{displaymath}
F(\delta)-F(\delta_0)=\int_{\delta_0}^\delta f(t)\dt,
\end{displaymath}
where the density for arguments $t>0$ is given by
\begin{displaymath}
f(t)=\int_{\mathbb{R}^m} h(x,t)W_m(x)\dx.
\end{displaymath}
Because the exponential terms partially cancel 
each other out, it reads explicitly
\begin{displaymath}
f(t)=\int_{\mathbb{R}^m}\frac
{R(\|Ax\|^2-\|x\|^2)^{(n-m)/2-1}\|Ax\|^2}
{\Gamma((n-m)/2)\,t}\;W_m(Ax)\dx
\end{displaymath}
in terms of the matrix $A=\Sigma_0/\sqrt{t}$,
or, after a change of variables, 
\begin{displaymath}
f(t)=\int_{\mathbb{R}^m}\frac
{R(\|x\|^2-\|A^{-1} x\|^2)^{(n-m)/2-1}\|x\|^2}
{\Gamma((n-m)/2)\,t\,|\det A|}\;W_m(x)\dx.
\end{displaymath}
Using Lemma~\ref{lm3.1}, $|\det A|=t^{-m/2}\det\Sigma_0$,
and the representation
\begin{displaymath}   
B(\a,\b)=\frac{\Gamma(\a)\Gamma(\b)}{\Gamma(\a+\b)}
\end{displaymath}
of the beta function in terms of the gamma 
function, this volume integral can finally 
be converted into the surface integral 
(\ref{eq3.9}).

Since the function (\ref{eq3.9}) is locally 
integrable and the distribution function
(\ref{eq3.7}) is right-continuous, one can 
let $\delta_0$ tend to zero in the relation 
above. This leads to
\begin{displaymath}
F(\delta)-F(0)=\int_0^\delta\!f(t)\dt
\end{displaymath}
for $\delta>0$. Since the kernel of the 
$(m\times n)$-matrix $T^t$ is as an 
$(n-m)$-dimensional subspace a set of 
volume measure zero, $F(0)=0$ follows 
from the representation (\ref{eq3.7}) of 
the distribution. Therefore, as stated, 
the function (\ref{eq3.9}) is the density 
of the distribution (\ref{eq3.2}) and
(\ref{eq3.3}), respectively.

Let $\sigma_m=\|T\|$ be the maximum 
singular value of $T$ and let 
$\Sigma_0\eta_0=\sigma_m\eta_0$ for a point 
$\eta_0$ on the unit sphere $S^{m-1}$. Let 
$0<t<\sigma_m^2$. Because of
\begin{displaymath}
1-t\,\|\Sigma_0^{-1}\eta_0\|^2=1-t/\sigma_m^2>0,
\end{displaymath}
the inequality $1-t\,\|\Sigma_0^{-1}\eta\|^2>0$ 
holds for all $\eta$ in a neighborhood of  
$\eta_0$. This implies $f(t)>0$. As
$1-t\,\|\Sigma_0^{-1}\eta\|^2\leq 0$ for   
$t\geq\sigma_m^2$ and $\eta\in S^{m-1}$, 
$f(t)=0$ for $t\geq\sigma_m^2$.
\qed
\end{proof}

A function $\eta\to h(\|T^t\eta\|^2)$ is integrable 
over the unit sphere $S^{n-1}$ if and only if the 
function $t\to h(t)f(t)$ is integrable over the real 
axis. If this is the case, 
\begin{equation}    \label{eq3.11}
\frac{1}{n\nu_n}\int_{S^{n-1}}h\big(\|T^t\eta\|^2\big)\deta 
=\int_{-\infty}^\infty h(t)f(t)\dt
\end{equation}
holds and the high-dimensional integral over the 
unit sphere reduces to an integral in one space 
dimension. As follows from (\ref{eq3.9}), the 
quantities
\begin{equation}    \label{eq3.12}
\frac{1}{\|T^t\eta\|^s}, \quad s>0,
\end{equation}
are therefore integrable over the unit sphere
if and only if $s$ is less than $m$.

If the matrix $\Sigma_0=\sigma I_m$ is a scalar
multiple the identity matrix, the integral 
(\ref{eq3.9}) can be evaluated explicitly. In 
this case, the density is
\begin{equation}    \label{eq3.13}    
f(t)=\frac{1}{\sigma^2}\,f_0\left(\frac{t}{\sigma^2}\right), 
\quad
f_0(t)=\frac{1}{B(\a,\b)}\;t^{\a-1}(1-t)^{\b-1},
\end{equation}
on the interval $0<t<\sigma^2$ and vanishes 
outside this interval. That is, the distribution
\begin{equation}    \label{eq3.14}
F(\delta)=F_0\left(\frac{\delta}{\sigma^2}\right),
\quad
F_0(\delta)=\int_0^\delta\!f_0(t)\dt,
\end{equation}
is a rescaled variant of a beta distribution.
For high dimensions, it therefore behaves
almost like a normal distribution. Figure~1 
shows this distribution for the dimensions 
$m=2^k$, $k=1,\ldots,16$, and $n=2m$, and 
for $\sigma^2=n/m$. The higher the dimensions 
are, the more the distribution approaches 
the step function that jumps at $\delta_0=1$ 
from zero to one and the more the values 
$\|T^t\eta\|^2$, $\|\eta\|=1$, cluster 
around one. As will be seen, this is a 
typical behavior that is not restricted 
to this particular case. 

\begin{figure}[t]   \label{fig1}
\includegraphics[width=0.93\textwidth]{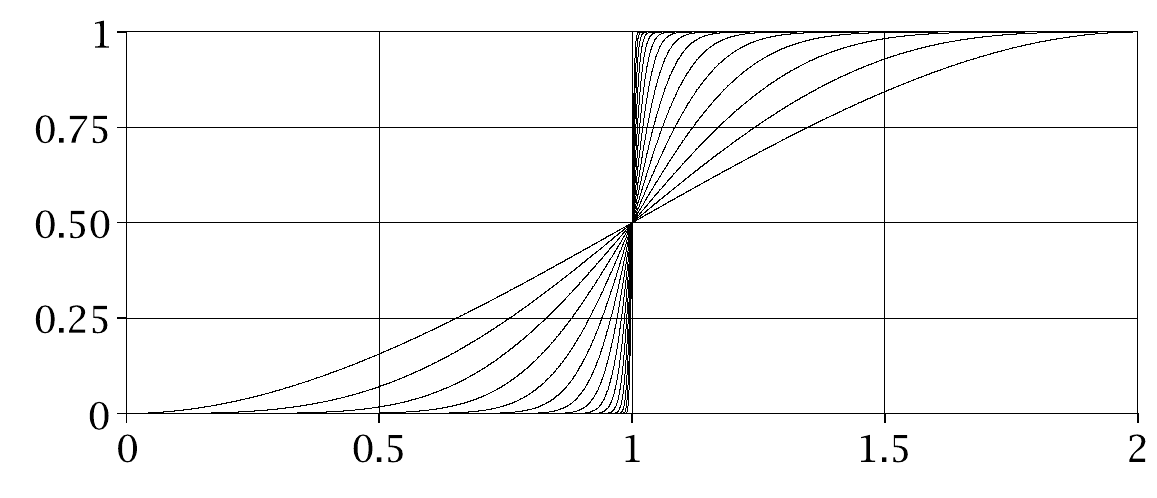}
\caption{The probability distributions 
 assigned to the densities \rmref{eq3.13} 
 for $m=2^k$, $k=1,\ldots,16$,  and $n=2m$}
\end{figure}

In the general case, the distribution (\ref{eq3.3})
of the values $\|T^t\eta\|^2$, $\eta\in S^{n-1}$,
can be expressed in terms of the beta distribution 
$F_0$. Introducing the probability measure
\begin{equation}    \label{eq3.15}
\mu(A)=\,\P_m\big(\big\{\eta\in S^{m-1}
\,\big|\,
\|\Sigma_0^{-1}\eta\|^2\!\in A\,\big\}\big),
\quad
\P_m(M)=\frac{1}{m\nu_m}\int_{M\cap S^{m-1}}\!\deta,
\end{equation}
on the Borel subsets $A$ of $\mathbb{R}$,
the density (\ref{eq3.9}) can be written as
\begin{equation}    \label{eq3.16}
f(t)=\frac{1}{B(\a,\b)}\int_0^\infty
\frac{t^{\a-1}R(1-s\,t)^{\b-1}}{\det\Sigma_0}\,
\dmu(s),
\end{equation}
that is, as a one-dimensional integral with 
respect to the probability measure $\mu$.

\begin{lemma}       \label{lm3.2}
For every measurable function 
$h:\mathbb{R}\to\mathbb{R}$ with values $h(t)\geq 0$, 
\begin{equation}    \label{eq3.17}    
\int_0^\infty h(t)f(t)\dt\,=\,
\frac{1}{\det\Sigma_0}\int_0^\infty\bigg(
\int_0^\infty \frac{h(t)f_0(s\,t)\,s}{s^\a}\,\dt
\bigg)\dmu(s).
\end{equation}
The same holds if the function $t\to h(t)f(t)$ 
is integrable.
\end{lemma}

\begin{proof}
The integral possesses the representation
\begin{displaymath}
\int_0^\infty h(t)f(t)\dt\,=\,
\frac{1}{B(\a,\b)}\int_0^\infty\bigg(\int_0^\infty
h(t)\,\frac{t^{\a-1}R(1-s\,t)^{\b-1}}{\det\Sigma_0}\,
\dmu(s)\bigg)\dt.
\end{displaymath}
For nonnegative $h$, Tonelli's theorem 
allows to swap the order of integration 
on the right-hand side, regardless of 
any integrability. The double integral 
then turns into
\begin{displaymath}
\frac{1}{B(\a,\b)}\int_0^\infty\bigg(\int_0^\infty
h(t)\,\frac{t^{\a-1}R(1-s\,t)^{\b-1}}{\det\Sigma_0}\,
\dt\bigg)\dmu(s)
\end{displaymath}
and the proposition is proved for nonnegative,
measurable functions $h$. Integrable functions 
are split into a negative and a positive part.
\qed
\end{proof}

With $h$ the characteristic function of 
the interval $[0,\delta]$, one gets the 
representation
\begin{equation}    \label{eq3.18}
\int_0^\delta\!f(t)\dt\,=\,
\frac{1}{\det\Sigma_0}
\int_0^\infty\frac{F_0(s\delta)}{s^\a}\,\dmu(s)
\end{equation}
of the distribution of the values $\|T^t\eta\|^2$
on the unit sphere of the $\mathbb{R}^n$.
Because of
\begin{equation}    \label{eq3.19}
1\,=\,\frac{1}{\det\Sigma_0}
\int_0^\infty\frac{1}{s^\a}\,\dmu(s),
\end{equation}
as follows inserting the constant function $h=1$,
this representation can be considered as a convex
combination of the rescaled beta distributions
\begin{equation}    \label{eq3.20}
F_0\left(\frac{\delta}{\sigma^2}\right),
\;\; \sigma_1\leq\sigma\leq\sigma_m.
\end{equation}
If $\Sigma_0=\sigma I_m$, the measure $\mu$ is 
a point measure supported at $s=1/\sigma^2$.
However, in most cases the measure possesses 
an integrable function as density.

The values $\|T^t\eta\|^2$, $\|\eta\|=1$, in the 
example above strongly cluster around a single 
point, the more the higher the dimensions are.
This kind of behavior is the norm and is reflected 
in the variance of the distributions. They tend 
with very high probability to zero as the 
dimensions increase. The expected value and 
the variance  
\begin{equation}    \label{eq3.21}    
E=\frac{1}{n\nu_n}\int_{S^{n-1}}\!\|T^t\eta\|^2\deta, 
\quad
V=\frac{1}{n\nu_n}\int_{S^{n-1}}\!\big(\|T^t\eta\|^2-E\big)^2\deta
\end{equation} 
will play a fundamental role in the 
further considerations.

\begin{lemma}       \label{lm3.3}
The expected value and the variance \rmref{eq3.21} 
depend only on the singular values 
$\sigma_1,\ldots,\sigma_m$ of the matrix $T$. 
In terms of the power sums
\begin{equation}    \label{eq3.22}
A_1=\sum_{i=1}^m\sigma_i^2, \quad
A_2=\sum_{i=1}^m\sigma_i^4
\end{equation}
of order one and two of the squares of the
singular values, they read as follows
\begin{equation}    \label{eq3.23}
E=\frac{A_1}{n}, \quad 
V=\frac{2n\,A_2-2A_1^2}{n^2(n+2)}.
\end{equation}   
\end{lemma}

\begin{proof}
Since the distribution, and with that its  moments, 
depend only on the singular values of $T$, we can 
again assume that $T^t=(\Sigma_0\;0)$ is a diagonal 
matrix with the singular values on its diagonal. 
Lemma~\ref{lm3.1} then leads to the representation
\begin{displaymath}
\frac{1}{n\nu_n}\int_{S^{n-1}}\big(\|T^t\eta\|^2)^k\deta
=C(2k)\int_{\mathbb{R}^n}
\bigg(\sum_{i=1}^m\!\sigma_i^2\omega_i^2\bigg)^k 
W(\omega)\domega
\end{displaymath}
of the moment of order $k$ as a homogeneous symmetric 
polynomial of degree $k$ in the~$\sigma_i^2$. The 
volume integral on the right-hand side splits into 
a sum of products of one-dimensional integrals. In 
principle, it can be calculated this way.
For the moments of order one and two, this is
possible without problems. For higher-order moments, 
the number of terms to be considered separately
increases rapidly. One can then take advantage of the 
fact that the symmetric polynomials are polynomials 
in the power sums of the $\sigma_i^2$; see 
\cite{VanDerWarden}, or \cite{Sturmfels} for 
a more recent treatment.
\qed
\end{proof}

In terms of the normalized singular values 
$\eta_i=\sigma_i/\sqrt{n}$, the expected value and 
the variance  (\ref{eq3.21}) and (\ref{eq3.23}), 
respectively, can be written as follows
\begin{equation}    \label{eq3.24}
E=\sum_{i=1}^m\eta_i^2, \quad
V=\frac{2n}{n+2}\,\sum_{i=1}^m\eta_i^4 \,-\,
\frac{2}{n+2}\,\bigg(\sum_{i=1}^m\eta_i^2\bigg)^2.
\end{equation}
We are interested in matrices $T$ for which the 
expected value $E$ is one, that is, for which the 
vector $\eta$ composed of the normalized singular 
values $\eta_i$ lies on the unit sphere of the 
$\mathbb{R}^m$. The variances $V$ then possess 
the representation
\begin{equation}    \label{eq3.25}
V=\frac{2n}{n+2}\,X(\eta)\,-\,\frac{2}{n+2}, 
\quad X(\eta)=\sum_{i=1}^m\eta_i^4.
\end{equation}
The function $X$ attains the minimum value $1/m$
and the maximum value one on the unit sphere of 
the $\mathbb{R}^m$. The variances (\ref{eq3.25}) 
therefore extend over the interval
\begin{equation}    \label{eq3.26}
\frac{n-m}{n+2}\,\frac{2}{m}\leq V<\frac{2n-2}{n+2}.
\end{equation}
However, they are most likely of 
the order $\mathcal{O}(1/m)$.

\begin{lemma}       \label{lm3.4}
Let the vectors $\eta$ composed of the normalized 
singular values $\eta_i$ be uniformly distributed 
on the part of the unit sphere consisting of 
points with strictly positive components. Then 
the expected value and the variance of $X$ are   
\begin{equation}    \label{eq3.27}
\mathbb{E}(X)=\frac{3}{m+2}, \quad 
\mathbb{V}(X)=\frac{24m-24}{(m+2)^2(m+4)(m+6)}.
\end{equation}
\end{lemma}

\begin{proof}
For symmetry reasons and because the intersections 
of lower-dimensional subspaces with the unit sphere 
have measure zero, we can allow points $\eta$ that 
are uniformly distributed on the whole unit sphere. 
The expected value and the variance of the function 
$X$, treated as a random variable, are therefore
\begin{displaymath}    
\mathbb{E}(X)=\frac{1}{m\nu_m}\int_{S^{m-1}}X(\eta)\deta,
\quad \mathbb{V}(X)=
\frac{1}{m\nu_m}\int_{S^{m-1}}\big(X(\eta)-\mathbb{E}(X)\big)^2\deta
\end{displaymath}
and can be calculated along the lines given 
by Lemma~\ref{lm3.1}. 
\qed
\end{proof}
It is instructive to express the variance
(\ref{eq3.27}) in terms of the random variable
\begin{equation}    \label{eq3.28}
\widetilde{X}(\eta)=\frac{m+2}{3}\,X(\eta),
\end{equation}
which is rescaled to the expected value
$\mathbb{E}(\widetilde{X})=1$. Its variance
\begin{equation}    \label{eq3.29}   
\mathbb{V}(\widetilde{X})=
\frac{8}{3}\,\frac{m-1}{(m+4)(m+6)}
\end{equation}
tends to zero as $m$ goes to infinity. This 
not only means that the expected value 
\begin{equation}    \label{eq3.30}
V^*=\,\frac{2n}{n+2}\,\frac{3}{m+2}\,-\,\frac{2}{n+2}
\end{equation}
of the variances (\ref{eq3.25}) tends to zero 
as the dimension $m$ increases, but also that 
the variances increasingly cluster around their 
expected value as $m$ increases.

The last observation is supported by simple 
experiments. Uniformly distributed points on 
the unit sphere $S^{m-1}$ can be generated 
from vectors in $\mathbb{R}^m$ with independent, 
standard normally distributed components. 
Such vectors themselves follow the standard 
normal distribution in the $m$-dimensional 
space. Scaling them to length one gives the 
desired uniformly distributed points on the 
unit sphere. This allows one to sample the 
random variable $X$ for any given 
dimension $m$.

The expected value and the variance (\ref{eq3.21}) 
can be expressed directly in terms of the entries 
of the $(m\times m)$-matrix $S=T^tT$, since the 
power sums (\ref{eq3.22}) are the traces
\begin{equation}    \label{eq3.31}
A_1=\sum_{i=1}^mS_{ii}, \quad
A_2=\sum_{i,j=1}^mS_{ij}^2
\end{equation}
of the matrices $S$ and $S^2$, and can therefore
be computed without recourse to the singular 
values of $T$.
Consider the $(n\times m)$-matrix $T$ assigned 
to an arbitrarily given undirected graph with 
$m$ vertices and $n-m$ edges that maps the 
components $x_i$ of a vector $x\in\mathbb{R}^m$ 
first to themselves and then to the $n-m$ 
weighted differences\footnote{The square root 
is important. Why, is explained in Sect.~\ref{sec4}.}
\begin{equation}    \label{eq3.32}
\frac{x_i-x_j}{\sqrt{2}}, \quad i<j,
\end{equation}
assigned to the edges of the graph connecting the 
vertices $i$ and $j$. In quantum physics, matrices 
of the given kind can be associated with the 
interaction of particles. In the given case, 
the matrix $S$ has the form $S=I+L/2$, where 
$I$ is the identity matrix and $L$ is the Laplacian 
matrix of the graph. The off-diagonal entries of 
$L$ are $L_{ij}=-1$ if the vertices $i$ and $j$ 
are connected by an edge and $L_{ij}=0$ otherwise. 
The diagonal entries $L_{ii}=d_i$ are the vertex 
degrees, the numbers of edges that meet at the 
vertices $i$.

\begin{lemma}       \label{lm3.5}
For a matrix $T$ that is assigned to an undirected 
graph with $m$ vertices as described above, the 
expected value and the variance \rmref{eq3.21} are
\begin{equation}    \label{eq3.33}
E=1,\quad 
V=\frac{2g_2+6g_1}{(g_1^2+4g_1+4)\,m+(4 g_1+8)},
\end{equation}
where $g_1$ and $g_2\geq g_1^2$ are the mean 
values of the vertex degrees and their squares:
\begin{equation}    \label{eq3.34}
g_1=\frac{1}{m}\sum_{i=1}^m d_i, \quad
g_2=\frac{1}{m}\sum_{i=1}^m d_i^2.
\end{equation}
\end{lemma}

\begin{proof}
By (\ref{eq3.31}), in the given case the constants 
(\ref{eq3.22}) possess the representation
\begin{displaymath}
A_1=\frac{g_1+2}{2}\,m,\quad
A_2=\frac{g_2+5g_1+4}{4}\,m
\end{displaymath}
and it is $n-m=mg_1/2$. Therefore, the proposition 
follows from (\ref{eq3.23}).
\qed
\end{proof}

For the matrices assigned to a family of graphs
for which the above two mean values remain bounded 
independent of the number of vertices, the variance 
tends to zero as the number of vertices goes to 
infinity. The matrices assigned to graphs whose 
vertices up to one are connected with a designated 
central vertex, but not with each other, form the 
other extreme. For these matrices, the variances 
decrease to a limit value greater than zero as 
the number of vertices increases. However, such 
matrices are a rare exception, not only with 
respect to the above random matrices, but also 
in the context of matrices assigned to graphs. 
Because of
\begin{equation}    \label{eq3.35}
V<\frac{2g_2}{g_1^2}\,\frac{1}{m},
\end{equation}
the inequality $V\geq\delta$ implies the lower
bound $g_2/g_1^2>m\delta/2$. Consider the random 
graphs with a fixed number $m$ of vertices that 
are with a given probability $p$ connected by 
an edge, or those with a correspondingly given 
number of edges. Sampling the variances assigned 
to a large number of graphs in such a class, one 
sees that these variances do not exceed the value 
$2/m$ with a high probability. 

The distribution of the values $\|T^t\eta\|^2$ 
for points $\eta$ on the unit sphere $S^{n-1}$
of the $\mathbb{R}^n$ depends only on the
singular values of the matrix $T$, regardless 
of its origin. Therefore, one may assume that 
$T^t=(\Sigma_0\;0)$ is a diagonal matrix, with 
the singular values of $T$ on the diagonal of 
the square matrix $\Sigma_0$.  Given the above 
remarks about generating such points, to sample 
the values $\|T^t\eta\|^2$ for a large number 
of independent, uniformly distributed points 
$\eta\in S^{n-1}$ then means to sample the 
ratio
\begin{equation}    \label{eq3.36}
\frac{\|\Sigma_0x\|^2}{\|x\|^2+\|y\|^2}
\end{equation}
for a large number of vectors $x\in\mathbb{R}^m$ 
and $y\in\mathbb{R}^{n-m}$ with independent and 
standard normally distributed components. The 
squares of the euclidean norms $\|y\|$ then follow 
the $\chi^2$-distribution with $n-m$ degrees of 
freedom. Therefore, their calculation can be 
replaced by the calculation of a single scalar 
quantity. The amount of work then remains 
proportional to the dimension $m$, no matter 
how much the two dimensions differ and how 
large their difference $n-m$ is.

With the exception of a few extreme cases, one 
observes that the distribution of the values 
$\|T^t\eta\|^2$, $\eta\in S^{n-1}$, more and 
more approaches the normal distribution with the 
expected value and the variance (\ref{eq3.21}) 
as the dimensions increase, a fact that underlines 
the importance of these quantities. This can be 
verified by comparing the frequency distribution 
of these values with the corresponding Gauss 
function
\begin{equation}    \label{eq3.37}
\frac{1}{\sqrt{2\pi V}}\,
\exp\left(-\,\frac{(\,t-E)^2}{2V}\,\right).
\end{equation}
Let $T$ be the matrix assigned to the graph 
associated with the $\mathrm{C}_{60}$-fullerene 
molecule, which consists of the ninety edges of 
a truncated icosahedron and its sixty corners 
as vertices. The degree of these vertices is 
three and the assigned variance therefore 
$V=9/380$. As for any matrix assigned to a 
graph, the expected value is $E=1$. The 
frequency distribution of the values 
$\|T^t\eta\|^2$ for a million randomly 
chosen points $\eta$ on the unit sphere and 
the assigned Gauss function (\ref{eq3.37}) 
are shown in Fig.~2.

\begin{figure}[t]   \label{fig2}
\includegraphics[width=0.93\textwidth]{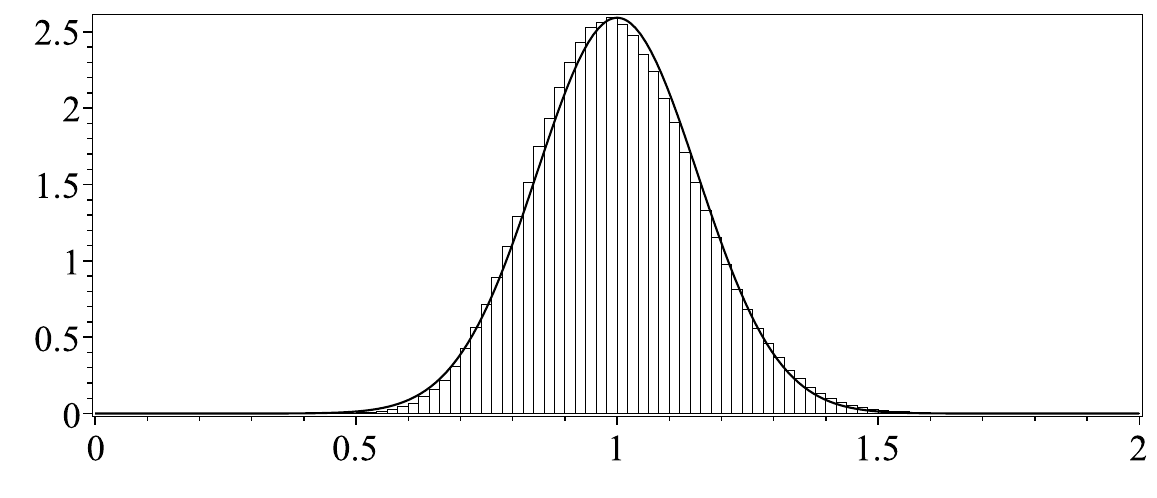}
\caption{The frequency distribution of the values 
$\|T^t\eta\|^2$ for the matrix $T$ associated 
with the $\mathrm{C}_{60}$-molecule}
\end{figure}


\setcounter{section}{3}

\section{The iterative procedure}
\label{sec4}

\setcounter{equation}{0}
\setcounter{theorem}{0}

\setcounter{figure}{2}

Now we are ready to analyze the iterative method
\begin{equation}    \label{eq4.1}
U_0=0, \quad
U_{k+1}=\,U_k\,-\,(-\Delta+\mu)^{-1}(\L U_k+\mu U_k-F)
\end{equation}
presented in the introduction and its polynomially 
accelerated counterpart, respectively, for the 
solution of the equation (\ref{eq2.9}), the 
equation that has replaced the original equation 
(\ref{eq2.11}). The iteration error possesses 
the Fourier representation
\begin{equation}    \label{eq4.2}
(\fourier{U}-\fourier{U}_k)(\omega)=
P_k\big(\alpha(\omega)(\|T^t\omega\|^2+\mu)\big)\fourier{U}(\omega),
\end{equation}
where $U$ is the exact solution (\ref{eq2.10}) 
of the equation (\ref{eq2.9}), $\alpha(\omega)$ 
is given by
\begin{equation}    \label{eq4.3}
\alpha(\omega)=\frac{1}{\|\omega\|^2+\mu},
\end{equation}
and the functions $P_k(\lambda)$ are polynomials 
of order $k$ with value $P_k(0)=1$. Throughout 
this section, we assume that the expression 
$\|T^t\eta\|^2$ possesses the expected value one. 
We restrict ourselves to the analysis of this 
iteration in the spaces of the functions with 
finite norms (\ref{eq2.17}). In parts, this 
analysis can be transferred to the Hilbert 
spaces $H^s$.

\begin{theorem}     \label{thm4.1}
If the solution $U$ possesses a finite norm $\|U\|_s$, 
this also holds for the iterates $U_k$. For all 
coefficients $\mu\geq 0$, the corresponding norm of 
the error, given by
\begin{equation}    \label{eq4.4}
\|U-U_k\|_s=\int\|T^t\omega\|^s\,
\big|P_k\big(\alpha(\omega)(\|T^t\omega\|^2+\mu)\big)\fourier{U}(\omega)\big|
\domega,
\end{equation}
then tends to zero for suitably chosen 
polynomials $P_k$ as $k$ goes to infinity.
\end{theorem}

\begin{proof}
Because the expression $\|T^t\eta\|^2$ possesses the 
expected value one, the spectral norm of the matrix
$T^t$ attains a value $\|T^t\|\geq 1$. 
If one sets $\vartheta=1/\|T^t\|^2$,
\begin{displaymath}
0\leq 1-\vartheta\alpha(\omega)(\|T^t\omega\|^2+\mu)<1
\end{displaymath}
therefore holds for all $\omega$ outside the kernel 
of $T^t$, as a subspace of a dimension less than $n$ 
a set of measure zero. Choosing 
$P_k(\lambda)=(1-\vartheta\lambda)^k$,
the proposition thus follows from the dominated 
convergence theorem.
\qed
\end{proof}

Of course, one would like to have more than 
just convergence. The goal are true error 
estimates. The next theorem is a first step 
in this direction.

\begin{theorem}     \label{thm4.2}
Let $0<\delta<1$, $a=\delta$, $b=\|T^t\|^2$, and
$\kappa=b/a$ and let 
\begin{equation}    \label{eq4.5}       
P_k(\lambda)=T_k\bigg(\dfrac{b+a-2\lambda}{b-a}\bigg)
\bigg/T_k\bigg(\dfrac{b+a}{b-a}\bigg)
\end{equation} 
be the to the interval $a\leq\lambda\leq b$ transformed 
Chebyshev polynomial $T_k$ of the first kind of degree 
$k$. The norm of the iteration error \rmref{eq4.2} 
then satisfies the estimate
\begin{equation}    \label{eq4.6}   
\|U-U_k\|_s\leq
\frac{2q^k}{1+q^{2k}}\,\|U\|_s+\|U-U_\delta\|_s,
\quad
q=\frac{\sqrt{\kappa}-1}{\sqrt{\kappa}+1},
\end{equation}
where $\fourier{U}_\delta$ takes the same values as 
$\fourier{U}$ on the set of all $\omega$ for which
\begin{equation}    \label{eq4.7}
\|T^t\omega\|^2>\delta\|\omega\|^2
\end{equation}
holds and vanishes outside this set.
\end{theorem}

\begin{proof}
It is $a\leq\alpha(\omega)(\|T^t\omega\|^2+\mu)\leq b$
if (\ref{eq4.7}) holds and 
$0\leq\alpha(\omega)(\|T^t\omega\|^2+\mu)\leq b$ 
on the entire $\mathbb{R}^n$. Therefore, the 
proposition follows from the estimates
\begin{displaymath}
|P_k(\lambda)|\leq 1
\text{\;for $0\leq \lambda\leq a$},
\quad
|P_k(\lambda)|\leq\frac{2q^k}{1+q^{2k}}
\text{\;for $a\leq \lambda\leq b$}
\end{displaymath}
for the values of the Chebyshev polynomial 
(\ref{eq4.5}) on the interval $0\leq\lambda\leq b$.
\qed
\end{proof}
Depending on the size of $\kappa=\|T^t\|^2/\delta$,
the norm of the error $U-U_k$ soon reaches the size 
of the norm of $U-U_\delta$. The idea behind the 
estimate (\ref{eq4.6}) is that in high space dimensions 
the condition (\ref{eq4.7}) is satisfied for nearly 
all $\omega$ and that the part $U-U_\delta$ of the 
solution $U$ is therefore negligible. Let 
$\phi:S^{n-1}\to\mathbb{R}$ be the integrable 
function
\begin{equation}    \label{eq4.8}
\phi(\eta)=\,
n\nu_n\int_0^\infty\|T^t\eta\|^s\,|\fourier{U}(r\eta)|\,r^{s+n-1}\dr
\end{equation}
from the definition (\ref{eq2.16}) of the norms
and let $S(\delta)$ be the sector
\begin{equation}    \label{eq4.9}
S(\delta)=
\big\{\omega\,\big|\|T^t\omega\|^2\leq\delta\|\omega\|^2\big\}.
\end{equation}
With $M=S(\delta)\cap S^{n-1}$, then the remainder 
possesses the representation
\begin{equation}    \label{eq4.10}
\|U-U_\delta\|_s=\frac{1}{n\nu_n}\int_M\phi(\eta)\deta.
\end{equation}
If $\phi$ is square integrable, that is, if the
norm $\tnorm{U}_s$ of the solution $U$ is finite, 
the remainder can be estimated by means of the 
Cauchy-Schwarz inequality. Expressing the area 
measure $\P(M)$ in terms of the density 
(\ref{eq3.9}), this leads to 
\begin{equation}    \label{eq4.11}
\|U-U_\delta\|_s^2\leq\int_0^\delta f(t)\dt\;\tnorm{U-U_\delta}^2_s.
\end{equation}
The right-hand side of this inequality tends at 
least like $\mathcal{O}(\delta^{m/2})$ to zero 
as $\delta$ goes to zero. The speed of convergence 
doubles if the function (\ref{eq4.8}) remains 
bounded. 

The estimate (\ref{eq4.6}) is extremely robust 
in many respects. First, it is based on a 
pointwise estimate of the Fourier transform 
of the error. It is therefore equally valid 
for other Fourier-based norms. Second, the 
function (\ref{eq4.3}) can be replaced by 
any approximation $\widetilde{\alpha}(\omega)$ 
for which, with an appropriate $\varepsilon>0$,
an estimate
\begin{equation}    \label{eq4.12}
0\leq\widetilde{\alpha}(\omega)\leq(1+\varepsilon)\alpha(\omega)
\end{equation}
on the entire frequency space and an 
inverse estimate
\begin{equation}    \label{eq4.13}
(1-\varepsilon)\alpha(\omega)\leq\widetilde{\alpha}(\omega)
\end{equation}
on a sufficiently large spherical shell around 
its origin hold. To see why, note that
\begin{equation}    \label{eq4.14}
|\fourier{U}(\omega)|\leq\frac{1}{\delta}
\,\frac{|\fourier{F}(\omega)|}{\|\omega\|^2}
\end{equation}
holds for the frequency vectors $\omega$ outside 
the sector (\ref{eq4.9}). Therefore, the part of 
the solution associated with the 
$\omega\notin S(\delta)$ outside a sufficiently 
large ball around the origin can be neglected.
In the high dimensions considered, the same holds 
for the part of the solution associated with the 
$\omega$ in a small ball around the origin. The 
contribution of a ball $B(R)$ of radius $R$ 
around the origin to the norm 
(\ref{eq2.17}) is 
\begin{equation}    \label{eq4.15}
\int_{B(R)}\|T^t\omega\|^s\,|\fourier{U}(\omega)|\domega
\,=\,n\nu_n\int_0^R\psi(r)\,r^{s+n-1}\dr,
\end{equation}
where the function $\psi$ is the mean value
\begin{equation}    \label{eq4.16}
\psi(r)=\frac{1}{n\nu_n}\int_{S^{n-1}}
\|T^t\eta\|^s\,|\fourier{U}(r\eta)|\deta.
\end{equation} 
Provided the behavior of this mean value can be 
kept under control, the contribution of the ball 
rapidly decreases as soon as its radius falls 
below a certain bound.

Nevertheless, the estimate from Theorem~\ref{thm4.2}
is rather pessimistic, because it only takes into 
account the decay of the left tail of the distribution, 
but ignores the fast decay of its right tail. 
To see what can be reached, we study the behavior of 
the iteration in the limit $\mu=0$, the case where 
the underlying effects are most clearly brought to 
light. Decomposing the vectors $\omega=r\eta$ into 
a radial part $r\geq 0$ and an angular part 
$\eta\in S^{n-1}$, the error (\ref{eq4.2}) 
propagates frequency-wise as
\begin{equation}    \label{eq4.17}
(\fourier{U}-\fourier{U}_k)(r\eta)=
P_k\bigg(\frac{r^2\|T^t\eta\|^2+\mu}{r^2+\mu}\bigg)
\fourier{U}(r\eta),
\end{equation}
and after the transition to the limit value 
$\mu=0$ as
\begin{equation}    \label{eq4.18}
(\fourier{U}-\fourier{U}_k)(r\eta)=
P_k\big(\|T^t\eta\|^2\big)\fourier{U}(r\eta).
\end{equation} 
In the limit case, therefore, the method acts only 
on the angular part of the error. Nevertheless,
by Theorem~\ref{thm4.1} the iterates converge to 
the solution. To clarify the underlying effects, 
we prove this once again in a different form.

\begin{theorem}     \label{thm4.3}
If the solution $U$ possesses a finite norm $\|U\|_s$, 
this also holds for the iterates $U_k$ implicitly given 
by \rmref{eq4.18}. The norm \rmref{eq2.17} of the 
iteration error then tends to zero for suitably chosen 
polynomials $P_k$ as $k$ goes to infinity.
\end{theorem}

\begin{proof}
In radial-angular representation, the iteration 
error is given by
\begin{displaymath}
\|U-U_k\|_s=\frac{1}{n\nu_n}
\int_{S^{n-1}}\big|P_k\big(\|T^t\eta\|^2\big)\big|
\phi(\eta)\deta
\end{displaymath}
where the integrable function $\phi:S^{n-1}\to\mathbb{R}$ 
is defined by the expression
\begin{displaymath}
\phi(\eta)=\,
n\nu_n\int_0^\infty\|T^t\eta\|^s\,|\fourier{U}(r\eta)|\,r^{s+n-1}\dr.
\end{displaymath}
To prove the convergence of the iterates $U_k$ to 
the solution $U$, we again consider the polynomials 
$P_k(t)=(1-\vartheta t)^k$ with $\vartheta=1/\|T^t\|^2$. 
For $\eta\in S^{n-1}$, it is
\begin{displaymath}
0\leq 1-\vartheta\|T^t\eta\|^2\leq 1,
\end{displaymath}
where the value one is only attained on the
intersection of the $(n-m)$-dimensional kernel 
of the matrix $T^t$ with the unit sphere, 
that is, on a set of area measure zero. The 
convergence again follows from the dominated 
convergence theorem.
\qed
\end{proof}

Under a seemingly harmless additional assumption,
one obtains an estimate for the speed of convergence 
that can hardly be improved.
              
\begin{theorem}     \label{thm4.4}
Under the assumption that the norm $\tnorm{U}_s$ 
of the solution $U$ takes a finite value, the 
iteration error can be estimated as
\begin{equation}    \label{eq4.19}
\|U-U_k\|_s^2\leq 
\int_{-\infty}^\infty P_k(t)^2 f(t)\dt\;\tnorm{U}_s^2
\end{equation}
in terms of the density $f$ of the distribution 
of the values $\|T^t\eta\|^2$ on the unit sphere.
\end{theorem}

\begin{proof}
The proof is based on the same error representation
\begin{displaymath}
\|U-U_k\|_s=\frac{1}{n\nu_n}
\int_{S^{n-1}}\big|P_k\big(\|T^t\eta\|^2\big)\big|
\phi(\eta)\deta
\end{displaymath}
as that in the proof of the previous theorem, 
but by assumption the function 
\begin{displaymath}
\phi(\eta)=\,
n\nu_n\int_0^\infty\|T^t\eta\|^s\,|\fourier{U}(r\eta)|\,r^{s+n-1}\dr
\end{displaymath}
is now square integrable, not only integrable. 
Its correspondingly scaled $L_2$-norm is the 
norm $\tnorm{U}_s$ of the solution. The 
Cauchy-Schwarz inequality thus leads to 
\begin{displaymath}
\|U-U_k\|_s^2\leq 
\frac{1}{n\nu_n}\int_{S^{n-1}}\big|P_k(\|T^t\eta\|^2)\big|^2\deta
\;\tnorm{U}_s^2.
\end{displaymath}
If one rewrites the integral in terms of 
the density $f$, (\ref{eq4.19}) follows.
\qed
\end{proof}

In contrast to Theorem~\ref{thm4.1} and 
Theorem~\ref{thm4.2}, this theorem strongly 
depends on the involved norms. But of course 
one can hope that other error norms behave 
similarly. 

\begin{lemma}       \label{lm4.1}
Let the function $U$ lie in the Hilbert-space $H^s$ 
equipped with the seminorm $|\cdot|_{2,\,s}$. The 
iteration error, for rotationally symmetric functions 
$U$ given by
\begin{equation}    \label{eq4.20}
|U-U_k|_{2,\,s}^2=
\int_{-\infty}^\infty P_k(t)^2 f(t)\dt\;|U|_{2,\,s}^2,
\end{equation}
then tends to zero for suitably chosen polynomials 
$P_k$ as $k$ goes to infinity. 
\end{lemma}

\begin{proof}
In radial-angular representation, the iteration 
error is given by 
\begin{displaymath}
|U-U_k|_{2,\,s}^2=\frac{1}{n\nu_n}
\int_{S^{n-1}}\big|P_k\big(\|T^t\eta\|^2\big)\big|^2
\phi(\eta)\deta,
\end{displaymath}
where the integrable function $\phi:S^{n-1}\to\mathbb{R}$ 
is defined by the expression
\begin{displaymath}
\phi(\eta)\,=\,
n\nu_n\int_0^\infty|\fourier{U}(r\eta)|^2r^{2s+n-1}\dr.
\end{displaymath}
Convergence follows as in the proof of 
Theorem~\ref{thm4.3}. If $U$ is rotationally 
symmetric, the function $\phi$ takes the 
constant value $|U|_{2,\,s}^2$ and  
(\ref{eq4.20}) is proven.
\qed
\end{proof}
Because equality holds in (\ref{eq4.20}), 
this once again shows that little is lost 
in the error estimate (\ref{eq4.19}) and 
that the prefactors
\begin{equation}    \label{eq4.21}
\int_{-\infty}^\infty P_k(t)^2 f(t)\dt
\end{equation}
cannot really be improved and are de facto 
optimal.

The task is therefore to find the polynomials 
$P_k$ of order $k$ that minimize the integral 
(\ref{eq4.21}) under the constraint $P_k(0)=1$. 
These polynomials can be expressed in terms of 
the orthogonal polynomials assigned to the 
density~$f$ as weight function. Under the 
given circumstances, the expression 
\begin{equation}    \label{eq4.22}
(p,q)=\int_{-\infty}^\infty p(t)q(t)f(t)\dt
\end{equation}
defines an inner product on the space of the 
polynomials. Let the polynomials $p_k$ of 
order $k$ satisfy the orthogonality condition 
$(p_k,p_\ell)=\delta_{k\ell}$. 

\begin{lemma}       \label{lm4.2}
In terms of the given orthogonal polynomials $p_k$,
the polynomial $P_k$ of order $k$ that minimizes 
the integral \rmref{eq4.21} under the constraint 
$P_k(0)=1$ is
\begin{equation}    \label{eq4.23}
P_k(t)=M_k\sum_{j=0}^kp_j(0)p_j(t), \quad
\frac{1}{M_k}=\sum_{j=0}^kp_j(0)^2,
\end{equation}
and the integral itself takes the minimum value
\begin{equation}    \label{eq4.24}
\int_{-\infty}^\infty P_k(t)^2 f(t)\dt=M_k.
\end{equation}
\end{lemma}

\begin{proof}
We represent the optimum polynomial $P_k$ as 
linear combination
\begin{displaymath}
P_k(t)=\sum_{j=0}^k x_jp_j(t).
\end{displaymath}
The zeros of $p_j$ lie strictly between the 
zeros of $p_{j+1}$, the interlacing property 
of the zeros of orthogonal polynomials. The 
polynomials $p_0,p_1,\ldots,p_k$ therefore 
cannot take the value zero at the same time.
Introducing the vector $x$ with the 
components~$x_j$, the vector $p\neq 0$ with 
the components $p_j(0)$,  and the vector 
$a=p/\|p\|$, we have to minimize $\|x\|^2$ 
under the constraint $a^tx=1/\|p\|$. 
Because of
\begin{displaymath}
\|x\|^2=(a^tx)^2+\|x-(a^tx)a\|^2,
\end{displaymath}
the polynomial $P_k$ minimizes the integral
if and only if its coefficient vector $x$ 
is a scalar multiple of $a$ or $p$ that 
satisfies the constraint.
\qed
\end{proof}

In the previous section, we have seen that the 
values $\|T^t\eta\|^2$ are approximately normally 
distributed, with a variance $V$ and a standard 
deviation $\sigma=\sqrt{V}$ that tend to zero in 
almost all cases as the dimensions increase. 
This justifies replacing the actual distribution 
(\ref{eq3.2}) by the corresponding normal 
distribution with the density
\begin{equation}    \label{eq4.25}
f(t)=\frac{1}{\sqrt{2\pi}\,\sigma}\,
\exp\left(-\,\frac{(t-1)^2}{2\sigma^2}\,\right).
\end{equation}
Then one ends up up with a classical case and can 
express the orthogonal polynomials $p_k$ in 
terms of the Hermite polynomials $He_0(x)=1$, 
$He_1(x)=x$, and
\begin{equation}    \label{eq4.26}
He_{k+1}(x)=x\,He_k(x)-k\,He_{k-1}(x), \quad k\geq 1,
\end{equation}
that satisfy the orthogonality condition
\begin{equation}    \label{eq4.27}
\int_{-\infty}^\infty He_k(x)He_\ell(x)\e^{-x^2/2}\dx
 = \sqrt{2\pi}\,k!\,\delta_{kl}.
\end{equation}
In dependence of the standard deviation $\sigma$, 
in this case the $p_k$ are given by
\begin{equation}    \label{eq4.28}
p_k(t)=\frac{1}{\sqrt{k!}}\,He_k\bigg(\frac{t-1}{\sigma}\bigg).
\end{equation}
The first twelve $M_k=M_k(\sigma)$ assigned 
to the orthogonal polynomials (\ref{eq4.28}) 
for the standard deviations $\sigma=1/16$, 
$\sigma=1/32$, $\sigma=1/64$, and $\sigma=1/128$ 
are compiled in Table~1. They give a good 
impression of the speed of convergence that 
can be expected and are fully in line with 
our predictions.

\begin{table}[t]    \label{table1}
\vspace{1.25ex}           
\begin{center}
\renewcommand{\tabcolsep}{10pt}
\renewcommand{\arraystretch}{1.2}
\begin{tabular}{c c c c}
\toprule 
$\sigma=1/16$   &  $\sigma=1/32$   &  $\sigma=1/64$   &  $\sigma=1/128$ \\
\midrule  
 3.891051e-03   &   9.756098e-04   &   2.440810e-04   &   6.103143e-05  \\
 3.051618e-05   &   1.907343e-06   &   1.192093e-07   &   7.450581e-09  \\
 3.618180e-07   &   5.604306e-09   &   8.737544e-11   &   1.364492e-12  \\
 5.765467e-09   &   2.199911e-11   &   8.543183e-14   &   3.332296e-16  \\
 1.157645e-10   &   1.081565e-13   &   1.044654e-16   &   1.017371e-19  \\
 2.812071e-12   &   6.393485e-16   &   1.533625e-19   &   3.727770e-23  \\
 8.035187e-14   &   4.418030e-18   &   2.627999e-22   &   1.593745e-26  \\
 2.645905e-15   &   3.496016e-20   &   5.149158e-25   &   7.788140e-30  \\
 9.884788e-17   &   3.118421e-22   &   1.135567e-27   &   4.282074e-33  \\
 4.138350e-18   &   3.096852e-24   &   2.783941e-30   &   2.616285e-36  \\
 1.922373e-19   &   3.389755e-26   &   7.511268e-33   &   1.758579e-39  \\
 9.827520e-21   &   4.055803e-28   &   2.211916e-35   &   1.289675e-42  \\
\bottomrule   
\end{tabular}
\vspace{6pt}
\caption{The reduction factors $M_k(\sigma)$ assigned to 
 the  approximate density (\ref{eq4.25}) for the first 
 $12$ iterations}
\end{center}
\end{table}

The only matrices $T^t$ for which the distribution 
density (\ref{eq3.9}) is explicitly known and 
available for comparison are the rescaled versions 
of the orthogonal projections from $\mathbb{R}^n$ 
to $\mathbb{R}^m$, $(m\times n)$-matrices with one 
as the only singular value. By Theorem~\ref{thm3.1}, 
the densities (\ref{eq3.9}) assigned to such 
orthogonal projections take the value
\begin{equation}    \label{eq4.29}
f_0(t)=\frac{1}{B(\a,\b)}\;t^{\a-1}(1-t)^{\b-1}
\end{equation}
on the interval $0<t<1$ and vanish outside of 
it, where $\a$ and $\b$ in this case are given 
by (\ref{eq3.10}). The to the expected value 
one rescaled counterpart of such densities is 
\begin{equation}    \label{eq4.30}
f(t)=\frac{\a}{\a+\b}\,f_0\bigg(\frac{\a\,t}{\a+\b}\bigg).
\end{equation}
These densities can be transformed to the weight 
functions associated with Jacobi polynomials. The  
orthogonal polynomials assigned to them  can thus 
be expressed in terms of Jacobi polynomials. The 
details can be found in the appendix. For smaller 
dimensions, the resulting values (\ref{eq4.24}) 
tend to zero much faster than the values that one 
gets approximating the actual densities by the 
densities of the corresponding normal distributions. 
This is likely due the behavior of the given beta 
distributions for small $\delta$, a behavior that 
is common to all distributions (\ref{eq3.2}) by 
Theorem~\ref{thm3.1}. However, the values rapidly 
approach each other as the dimensions increase. 

For the matrices $T$ from Sect.~\ref{sec3} assigned 
to undirected interaction graphs, the angular 
distribution of the values $\|T^t\omega\|^2$ is 
often very similar to that in the case of rescaled 
orthogonal projections. In fact, it differs in many 
cases only very slightly from a single, to the 
expected value one rescaled beta distribution. 
The variance of a rescaled beta distribution with 
the density (\ref{eq4.30})  is
\begin{equation}    \label{eq4.31}    
V=\frac{\b}{\a(\a+\b+1)}.
\end{equation}
The parameters $\a$ and $\b$ and the variance $V$
are connected via the relation
\begin{equation}    \label{eq4.32}    
(1-V\a)\b=V\a(\a+1).
\end{equation}
Thus, such a beta distribution with given 
parameter $\a$ and given variance $V$ exists 
if and only if $1-V\a>0$. In this case, it 
is uniquely determined and $\b$ is given by
\begin{equation}    \label{eq4.33}    
\b=\frac{V\a(\a+1)}{1-V\a}.
\end{equation}
If one sets $\a$ to the value $m/2$ given 
by Theorem~\ref{thm3.1}, the compatibility 
condition
\begin{equation}    \label{eq4.34}    
V<\frac{2}{m}
\end{equation}
is satisfied for the matrices $T$ assigned to 
large classes of graphs, without exception or 
with high probability. This follows from the 
representation (\ref{eq3.33}) of the variance
in terms of the vertex degrees. Examples are 
the matrices assigned to regular graphs, whose 
vertices all have the same number of neighbors, 
or the matrices assigned to random graphs with 
a fixed number of vertices and edges. In such 
cases, the resulting densities (\ref{eq4.30})
are almost identical to the actual densities.
The values (\ref{eq4.24}) assigned to them 
therefore reflect the convergence behavior 
presumably at least as well as the values 
resulting from the given approximation by a 
simple Gauss function. Figure~3 compares the 
distribution density originating from the 
matrix assigned to a small random graph with 
$m=32$ vertices and $n-m=96$ edges with its 
approximation by the density of the 
corresponding normal distribution and that 
by the density of the given rescaled beta 
distribution. The variance is $V=941/16640$ 
in this example. The difference is striking, 
but soon disappears for larger graphs.

\begin{figure}[t]   \label{fig3}
\includegraphics[width=0.93\textwidth]{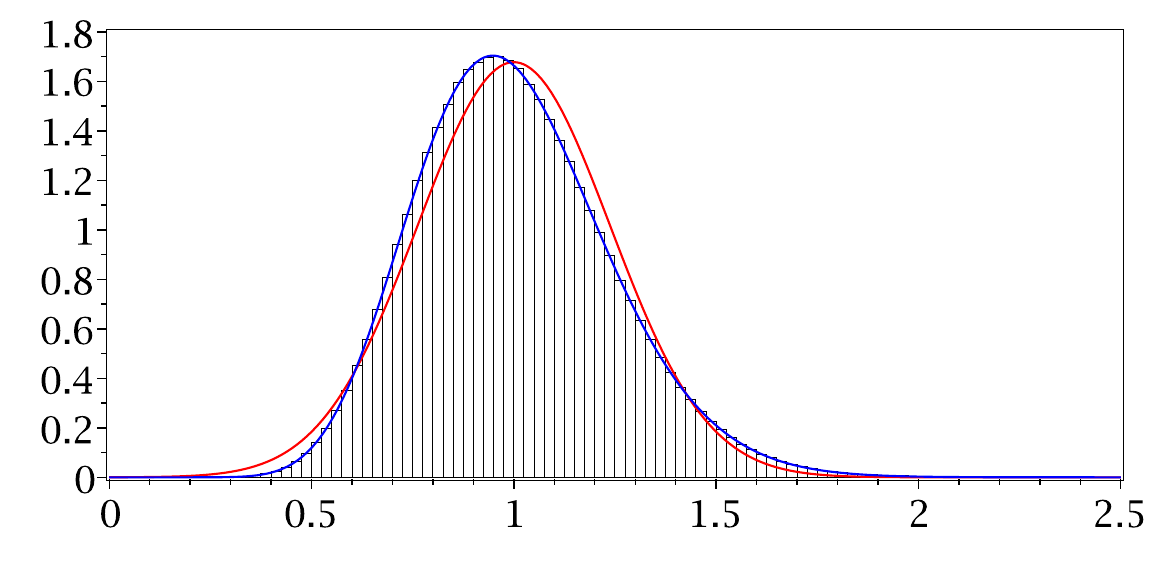}
\caption{Comparison of the approximation by a 
rescaled beta distribution and a normal 
distribution}
\end{figure}

The matrices $T$ assigned to interaction graphs 
still have another nice property. For vectors 
$e$ pointing into the direction of a coordinate 
axis, $\|T^te\|=\|e\|$ holds. The expression 
$\alpha(\omega)(\|T^t\omega\|^2+\mu)$ thus takes 
the value one on the coordinate axes, which is 
the reason for the square root in the definition 
(\ref{eq3.32}) of these matrices. The described 
effects become therefore particularly noticeable
on the regions on which the Fourier transforms 
of functions in hyperbolic cross spaces are 
concentrated. The functions that are well 
representable in tensor formats and in which 
we are first and foremost interested in 
the present paper fall into this category.

We still need to approximate $1/r$ on intervals 
$\mu\leq r\leq R\mu$ with moderate relative 
accuracy by sums of exponential functions, which 
then leads to the approximations (\ref{eq1.4}) 
of the kernel (\ref{eq4.3}) by sums of Gauss 
functions. Relative, not absolute accuracy, 
\pagebreak
since these approximations are embedded in an 
iterative process; see the remarks following 
Theorem~\ref{thm4.2}. It suffices to restrict 
oneself to intervals $1\leq r\leq R$. If $v(r)$ 
approximates $1/r$ on this interval with a 
given relative accuracy, the function
\begin{equation}    \label{eq4.35}
r\;\to\;\frac{1}{\mu}\,v\left(\frac{r}{\mu}\right)
\end{equation}
approximates the function $r\to 1/r$ on the original 
interval $\mu\leq r\leq R\mu$ with the same relative 
accuracy. Good approximations of $1/r$ with 
astonishingly small relative error are the at 
first sight rather harmless looking sums
\begin{equation}    \label{eq4.36}
v(r)=h\sum_{k=k_1}^{k_2}\e^{-kh}\exp(-\e^{-kh}r),
\end{equation}
a construction due to Beylkin and Monz\'{o}n
\cite{Beylkin-Monzon} based on an integral
representation. The parameter $h$ determines
the accuracy and the quantities $k_1h$ and $k_2h$ 
control the approximation interval. The functions 
(\ref{eq4.36}) possess the representation
\begin{equation}    \label{eq4.37}
v(r)=\frac{\phi(\ln r)}{r}, \quad
\phi(s)=\,h\sum_{k=k_1}^{k_2}\varphi(s-kh),
\end{equation}
in terms of the for $s$ going to infinity 
rapidly decaying window function
\begin{equation}    \label{eq4.38}
\varphi(s)=\exp(-\e^s+s).
\end{equation}
To determine with which relative error the function
(\ref{eq4.36}) approximates $1/r$ on a given 
interval $1\leq r\leq R$, thus one must check
how well the function $\phi$ approximates the 
constant $1$ on the interval $0\leq s\leq\ln R$. 

For $h=1$ and summation indices $k$ ranging 
from $-2$ to $50$, the relative error is, 
for example, less than $0.0007$ on almost 
the whole interval $1\leq r\leq 10^{18}$, 
that is, in the per mill range on an interval 
that spans eighteen orders of magnitude. 
Such an accuracy is surely exaggerated in 
the given context, but the example underscores 
the excellent approximation properties of 
the functions (\ref{eq4.36}).  Figure~4 
depicts the corresponding function $\phi$. 
These observations are underpinned by the 
analysis of the approximation properties 
of the corresponding infinite series in 
\cite[Sect. 5]{Scholz-Yserentant}. 
It is shown there that these series 
approximate $1/r$ on the positive real 
axis with a relative error
\begin{equation}    \label{eq4.39}
\sim 4\pi h^{-1/2}\e^{-\pi^2/h}.
\end{equation}

High relative accuracy on large intervals 
$1\leq r\leq R$ is a considerably stronger 
requirement than high absolute accuracy.
The approximation of $1/r$ with minimum 
absolute error has been studied in 
\cite{Braess-Hackbusch}, 
\cite{Braess-Hackbusch_2} and 
\cite{Hackbusch_2},
and the approximation both with minimum 
absolute and minimum relative error in
\cite{Braess-Hackbusch_3}.
The technique in \cite{Hackbusch_2} can be 
used to compute the approximations with 
the least relative error. On the given 
interval, one gains a factor of just over 
five with a sum of $50$ exponential functions. 
Hackbusch has compiled the best approximations 
for a large number of intervals and sums of 
exponential functions of very different 
length. The data can be downloaded from 
the website of the Max Planck Institute
in Leipzig \cite{Hackbusch_3}.

\begin{figure}[t]   \label{fig4}
\includegraphics[width=0.93\textwidth]{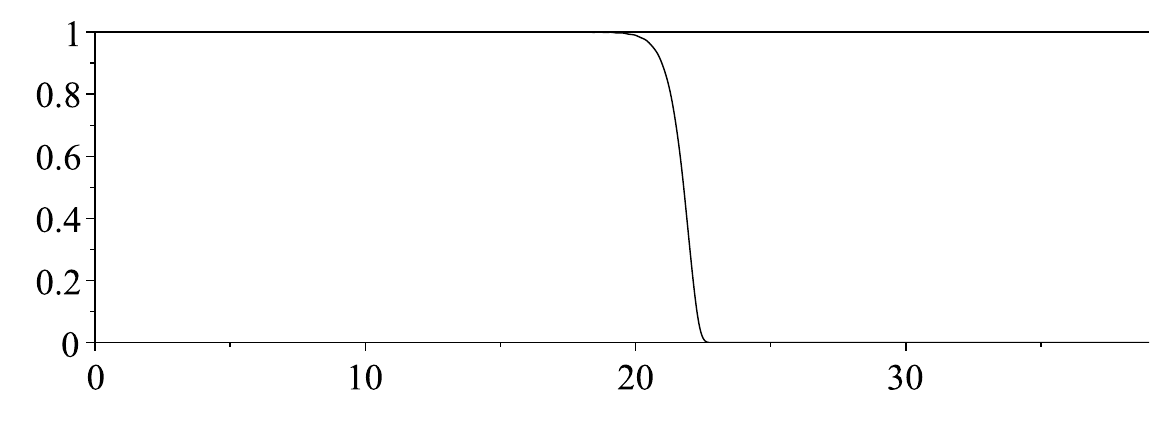}
\caption{The rescaled function $s\to\phi(s\ln 10)$ 
 approximating $1$ for $h=1$, $k_1=-2$, 
 and $k_2=50$}
\end{figure}

The practical feasibility of the approach depends 
on the representation of the tensors involved and 
the access to the Fourier transforms of the 
functions represented by them. A central task not 
discussed here is the recompression of the data 
between the iteration steps in order to keep the 
amount of work and storage under control, a problem 
common to all tensor-oriented iterative methods. 
If one fixes the accuracy in the single coordinate 
directions, the process reduces to an iteration on the 
space of the functions defined on a given high-dimensional 
cubic grid, functions that are stored in a compressed 
tensor format. There exist highly efficient, linear 
algebra-based techniques for recompressing such data. 
A problem in the given context may be that the in this 
framework naturally looking discrete $\ell_2$-norm of 
the tensors does not match the underlying norms of 
the continuous problem. Another open question is the 
overall complexity of the process. A difficulty 
with our approach is that the given operators do 
not split into sums of operators that act separately 
on a single variable or a small group of variables, 
a fact that complicates the application of techniques 
such as of those in \cite{Bachmayr-Dahmen} or 
\cite{Dahmen-DeVore-Grasedyck-Sueli}. For more 
information on tensor-oriented solution methods 
for partial differential equations, see the 
monographs \cite{Hackbusch_1} and \cite{Khoromskij} 
of Hackbusch and of Khoromskij and Bachmayr's 
comprehensive survey article \cite{Bachmayr}.


\bigskip

\noindent
{\bf Remark}

\medskip
\noindent
This is a largely rewritten version of the 
paper [Numer. Math. (2024) 156:777--811] 
of the same title. The central new result 
is Theorem~\ref{thm3.1} and, in its wake, 
the representation (\ref{eq3.18}) of the
probability distribution. Lemma~\ref{lm3.5} 
expresses the variance directly in terms 
of the vertex degrees. The present version
also uses two new scales of norms that 
directly measure the smoothness of the 
trace functions.


\section*{Appendix. Beta distributions and Jacobi polynomials}

Setting $a=\b-1$, $b=\a-1$, and introducing the constants
\begin{equation*}
K(k,a,b)=
\frac{2^{a+b+1}\,\Gamma(k+a+1)\Gamma(k+b+1)}
{(2k+a+b+1)\Gamma(k+1)\Gamma(k+a+b+1)},
\end{equation*}
the density (\ref{eq4.30}) can be rewritten in the form
\begin{equation*}
f(t)=\,2\,\frac{b+1}{a+b+2}\,
\frac{(1-x)^a(1+x)^b}{K(0,a,b)},\quad
x={}-1+2\,\frac{b+1}{a+b+2}\;t.
\end{equation*}
The orthogonal polynomials $p_k$ from Lemma~\ref{lm4.2}
assigned to this weight function~$f$ can therefore 
be expressed in terms of the Jacobi polynomials
\begin{equation*}
P_k^{(a,b)}(x)=P(k,a,b;x).
\end{equation*}
The Jacobi polynomials satisfy the orthogonality relation
\begin{equation*}
\int_{-1}^1P(k,a,b;x)P(\ell,a,b;x)(1-x)^a(1+x)^b\dx
=K(k,a,b)\delta_{k\ell}.
\end{equation*}
The polynomials $p_k$ therefore possess the representation
\begin{equation*}
p_k(t)=\bigg(\frac{K(0,a,b)}{K(k,a,b)}\bigg)^{1/2}
P\bigg(k,a,b;\,-1+2\,\frac{b+1}{a+b+2}\;t\,\bigg).
\end{equation*}
At $x=-1$, the Jacobi polynomials take the value
\begin{equation*}
P(k,a,b;\,-1)=
(-1)^k\frac{\Gamma(k+b+1)}{\Gamma(k+1)\Gamma(b+1)},
\end{equation*}
see \cite[Table 22.4]{Abramowitz-Stegun} 
or \cite[Table 18.6.1]{DLMF}. This leads 
to the closed representation
\begin{equation*}
p_k(0)^2=\,\frac
{(2k+a+b+1)\Gamma(a+1)\Gamma(k+b+1)\Gamma(k+a+b+1)}
{\Gamma(a+b+2)\Gamma(b+1)\Gamma(k+1)\Gamma(k+a+1)}
\end{equation*}
of the values $p_k(0)^2$. Starting from $p_0(0)^2=1$,
they can therefore be computed in a numerically very 
stable way by the recursion
\begin{equation*}
p_{k+1}(0)^2=\,\frac
{(k+b+1)(k+a+b+1)(2k+a+b+3)}
{(k+1)(k+a+1)(2k+a+b+1)}\;
p_k(0)^2.
\end{equation*}


\bibliographystyle{spmpsci}
\bibliography{references}

\begin{thebibliography}{10}
\providecommand{\url}[1]{{#1}}
\providecommand{\urlprefix}{URL }
\expandafter\ifx\csname urlstyle\endcsname\relax
  \providecommand{\doi}[1]{DOI~\discretionary{}{}{}#1}\else
  \providecommand{\doi}{DOI~\discretionary{}{}{}\begingroup
  \urlstyle{rm}\Url}\fi

\bibitem{Abramowitz-Stegun}
Abramowitz, M., Stegun, I.A.: Handbook of Mathematical Functions.
\newblock Dover Publications, New York (10th printing in 1972)

\bibitem{Bachmayr}
Bachmayr, M.: Low-rank tensor methods for partial differential equations.
\newblock Acta Numerica \textbf{32}, 1--121 (2023)

\bibitem{Bachmayr-Dahmen}
Bachmayr, M., Dahmen, W.: Adaptive near-optimal rank tensor approximation for
  high-dimensional operator equations.
\newblock Found. Comp. Math. \textbf{15}, 839--898 (2015)

\bibitem{Beylkin-Monzon}
Beylkin, G., Monz\'{o}n, L.: Approximation by exponential sums revisited.
\newblock Appl. Comput. Harmon. Anal. \textbf{28}, 131--149 (2010)

\bibitem{Braess-Hackbusch}
Braess, D., Hackbusch, W.: Approximation of $1/x$ by exponential sums in
  $[1,\infty)$.
\newblock IMA J. Numer. Anal. \textbf{25}, 685--697 (2005)

\bibitem{Braess-Hackbusch_2}
Braess, D., Hackbusch, W.: On the efficient computation of high-dimensional
  integrals and the approximation by exponential sums.
\newblock In: R.~DeVore, A.~Kunoth (eds.) Multiscale, Nonlinear and Adaptive
  Approximation. Springer, Berlin Heidelberg (2009)

\bibitem{Braess-Hackbusch_3}
Braess, D., Hackbusch, W.: On the approximation of {Stieltjes} functions by
  exponential sums and rational functions with applications to partial
  differential equations.
\newblock Numer. Math. \textbf{158}, 455--490 (2026).
\newblock See also
  \url{www.mis.mpg.de/publications/preprint-repository/article/2024/issue-10}
  for an updated version

\bibitem{Dahmen-DeVore-Grasedyck-Sueli}
Dahmen, W., DeVore, R., Grasedyck, L., S{\"u}li, E.: Tensor-sparsity of
  solutions to high-dimensional elliptic partial differential equations.
\newblock Found. Comp. Math. \textbf{16}, 813--874 (2016)

\bibitem{Hackbusch_3}
Hackbusch, W.: \url{www.mis.mpg.de/scicomp/EXP_SUM}

\bibitem{Hackbusch_2}
Hackbusch, W.: Computation of best {$L^\infty$} exponential sums for {$1/x$} by
  {Remez'} algorithm.
\newblock Computing and Visualization in Science \textbf{20}, 1--11 (2019)

\bibitem{Hackbusch_1}
Hackbusch, W.: Tensor Spaces and Numerical Tensor Calculus.
\newblock Springer, Cham (2019)

\bibitem{Khoromskij}
Khoromskij, B.N.: Tensor Numerical Methods in Scientific Computing, \emph{Radon
  Series on Computational and Applied Mathematics}, vol.~19.
\newblock De Gruyter, Berlin M{\"u}nchen Boston (2018)

\bibitem{DLMF}
Olver, F.W.J., Lozier, D.W., Boisvert, R.F., Clark, C.W. (eds.): {NIST}
  Handbook of Mathematical Functions.
\newblock Cambridge University Press, Cambridge (2010)

\bibitem{Scholz-Yserentant}
Scholz, S., Yserentant, H.: On the approximation of electronic wavefunctions by
  anisotropic {G}auss and {G}auss-{H}ermite functions.
\newblock Numer. Math. \textbf{136}, 841--874 (2017)

\bibitem{Sturmfels}
Sturmfels, B.: Algorithms in Invariant Theory.
\newblock Springer, Wien (2008)

\bibitem{VanDerWarden}
van~der Warden, B.L.: Algebra {I}.
\newblock Springer, Berlin Heidelberg New York (1971)

\bibitem{Yserentant_2020}
Yserentant, H.: On the expansion of solutions of {L}aplace-like equations into
  traces of separable higher-dimensional functions.
\newblock Numer. Math. \textbf{146}, 219--238 (2020)

\end{thebibliography}


\newpage

\section*{Supplement. Complete graphs and symmetry conservation}

\newenvironment{unnumberedlemma}{\vspace{2ex}\noindent\bf Lemma \it}{}

The goal in the background is the application 
of the here developed techniques to approximate 
inverse iterations for the calculation of the 
ground state of the electronic Schr\"odinger 
equation; see \cite{Scholz-Yserentant} for more 
details. Electronic wave functions are subject 
to the Pauli principle. Taking the spin into 
account, they are antisymmetric with respect 
to the exchange of the electrons, and if one 
considers the different spin components of the 
wave functions separately, they are antisymmetric 
with respect to the exchange of the electrons 
with the same spin. The following considerations 
show that such symmetry properties are preserved 
in the course of the calculations.

Since all electrons interact with each other, 
the matrix $T$ is in this case the matrix that 
is assigned to a complete graph, or better a 
triple copy of this matrix, one copy for each 
of the three spatial coordinates. We label the 
variables assigned to the vertices by the 
indices $i=1,\ldots,m$ and the variables 
assigned to the edges by the ordered index 
pairs $(i,j)$, $i<j$. The components of $Tx$ 
are in this notation
\begin{equation*}    
Tx|_i=x_i, \quad Tx|_{i,j}=\frac{x_i-x_j}{\sqrt{2}}.
\end{equation*}
To every permutation $\pi$ of the vertices  
$i=1,\ldots,m$ associated with the positions 
of the electrons, we assign two matrices, 
the $(m\times m)$-permutation matrix $P$ 
given by
\begin{equation*}    
Px|_i=x_{\pi(i)}, \quad i=1,\ldots,m,
\end{equation*}
and the much larger $(n\times n)$-matrix 
$Q$ given by
\begin{equation*}    
Qy|_i=y_{\pi(i)}
\end{equation*}
for the first $m$ components of $Qy$ associated 
with the vertices of the graph and by
\begin{equation*}   
Qy|_{ij}=
\begin{cases}
\; y_{\pi(i),\pi(j)}, &\text{if $\pi(i)<\pi(j)$} \\
{}-y_{\pi(j),\pi(i)}, &\text{otherwise}
\end{cases} 
\end{equation*}
for the remaining components associated 
with the edges. By construction then
\begin{equation*}    
 TP=QT
\end{equation*}
holds, which is the key to the following
considerations. As a permutation matrix,
the matrix $P$ is orthogonal. The same
holds for the matrix $Q$.

\begin{unnumberedlemma}       
Up to sign changes, the matrix $Q$ is 
a permutation matrix.
\end{unnumberedlemma}

\begin{proof}
Let the indices $i<j$ be given. 
If $\pi^{-1}(i)<\pi^{-1}(j)$, then
\begin{displaymath}
Qy|_{\pi^{-1}(i),\pi^{-1}(j)}=y_{ij}
\end{displaymath}
holds, and if $\pi^{-1}(j)<\pi^{-1}(i)$ one has
\begin{displaymath}
Qy|_{\pi^{-1}(j),\pi^{-1}(i)}={}-y_{ij}.
\end{displaymath}
All components of $y$ therefore also appear as
components of $Qy$, either with positive or 
negative sign, just in a different order.
\qed
\end{proof}
Furthermore, we assign to the matrix $Q$ the value
\begin{equation*}    
\epsilon(Q)=\mathrm{sign}(\pi),
\end{equation*}
that is, the number $\epsilon(Q)=+1$, if $\pi$ 
consists of an even number of transpositions, 
and the number $\epsilon(Q)=-1$ otherwise. 

Let $G$ be a subgroup of the symmetric group 
$S_m$, the group of the permutations of the 
indices $1,\ldots,m$. The matrices $P$ assigned 
to the elements of $G$ form a group which is
isomorphic to $G$ under the matrix multiplication 
as composition. We say that a function 
$u:\mathbb{R}^m\to\mathbb{R}$ is antisymmetric 
under the permutations in $G$, or for short
antisymmetric under $G$, if for all matrices 
$P=P(\pi)$ assigned to the $\pi\in G$
\begin{equation*}    
u(Px)=\mathrm{sign}(\pi)u(x)
\end{equation*}
holds. We say that a function
$U:\mathbb{R}^n\to\mathbb{R}$ is antisymmetric 
under $G$ if for all these permutations $\pi$ 
and the assigned matrices $Q=Q(\pi)$ 
\begin{equation*}    
U(Qy)=\epsilon(Q)U(y)
\end{equation*}
holds. Both properties correspond to
each other.

\begin{unnumberedlemma}       
If the function $U:\mathbb{R}^n\to\mathbb{R}$
is antisymmetric under $G$, so are its traces.
\end{unnumberedlemma}

\begin{proof}
This follows from the relation $TP=QT$ between 
the matrices. We have
\begin{displaymath}
U(TPx)=U(QTx)=\epsilon(Q)U(Tx).
\end{displaymath}
Because of $\epsilon(Q)=\mathrm{sign}(\pi)$,
the proposition follows.
\qed
\end{proof}
We say that a function $U:\mathbb{R}^n\to\mathbb{R}$
is symmetric under the permutations in $G$, or for
short symmetric under $G$, if for all matrices 
$Q=Q(\pi)$, $\pi\in G$,
\begin{equation*}   
U(Qy)=U(y)
\end{equation*}
holds. Since the matrices $Q(\pi)$ assigned to 
the permutations in $\pi\in G$ are orthogonal,
rotationally symmetric functions have this 
property.

In the following we show that for any right-hand 
side $F$ of the equation (\ref{eq2.9}), which is 
antisymmetric under the permutations in $G$, its 
solution (\ref{eq2.10}) is also antisymmetric 
under $G$ and that the same holds for the 
approximations of this solution. 

\begin{unnumberedlemma}       
If the function $F:\mathbb{R}^n\to\mathbb{R}$ 
possesses an integrable Fourier transform and
is antisymmetric under $G$, then for any 
measurable and bounded kernel $K$, which is 
symmetric under $G$, the functions
\begin{equation*}    
U(y)=\bigg(\frac{1}{\sqrt{2\pi}}\bigg)^n\!\int
K(\omega)\,\fourier{F}(\omega)\,
\mathrm{e}^{\,\mathrm{i}\,\omega\cdot y}\domega
\end{equation*}
are antisymmetric under the permutations 
in $G$.
\end{unnumberedlemma}

\begin{proof}
From the orthogonality of the matrices $Q=Q(\pi)$, 
$\pi\in G$, we obtain 
\begin{displaymath}    
U(Qy)=\bigg(\frac{1}{\sqrt{2\pi}}\bigg)^n\!\int
K(Q\omega)\,\fourier{F}(Q\omega)\,
\mathrm{e}^{\,\mathrm{i}\,\omega\cdot y}\domega
\end{displaymath}
and in the same way the representation
\begin{displaymath}    
\fourier{F}(Q\omega)=  
\bigg(\frac{1}{\sqrt{2\pi}}\bigg)^n\!\int
F(Qy)\,\mathrm{e}^{\,-\mathrm{i}\,\omega\cdot y}\dy.
\end{displaymath}
Because of $F(Qy)=\epsilon(Q)F(y)$, the Fourier 
transform of $F$ thus transforms like
\begin{displaymath}
\fourier{F}(Q\omega)=\epsilon(Q)\fourier{F}(\omega).
\end{displaymath}
As by assumption $K(Q\omega)=K(\omega)$, 
the proposition follows.
\qed
\end{proof}
The norm $\|\omega\|$ of the vectors $\omega$ in 
$\mathbb{R}^n$ is symmetric under the permutations 
in $G$, but also the norm of the vectors 
$T^t\omega$ in $\mathbb{R}^m$. This follows from 
the interplay of the matrices $P$, $Q$, and $T$, 
which is expressed in the relation $TP=QT$ 
and leads to
\begin{equation*}    
\|T^tQ\omega\|=\|P^tT^tQ\omega\|=
\|T^tQ^tQ\omega\|=\|T^t\omega\|.
\end{equation*}
This means that every kernel $K(\omega)$ which 
depends on the norms of $\omega$ and $T^t\omega$ 
only is symmetric under $G$. Provided that the 
right-hand side $F$ is antisymmetric under $G$, 
the same holds for the solution (\ref{eq2.10})
of the equation (\ref{eq2.9}) and for all its
approximations calculated as described in the
paper. 

In summary, the solution of the equation 
(\ref{eq2.9}) and all its approximations 
and their traces completely inherit the 
symmetry properties of the right-hand side. 
We conclude that our theory is fully 
compatible with the Pauli principle. What 
is still missing is a procedure for the 
recompression of the data between the single 
iteration steps that preserves the symmetry 
properties.


\end{document}